\documentclass[12pt,reqno]{amsart}
\usepackage{amsfonts}
\usepackage{indentfirst}
\usepackage{graphicx}
\usepackage{amssymb}
\usepackage{color,xcolor}
\usepackage{graphicx,epstopdf}
\usepackage{epsfig}
\usepackage{subfigure}
\usepackage{caption}
\usepackage{mathrsfs}
\usepackage{bbm}
\usepackage{chngpage}
\usepackage{cite}
\usepackage{multirow}
\usepackage{multicol}
\usepackage{dsfont}
\usepackage{bm}
\usepackage{amssymb,amsmath,amsthm,mathrsfs}%
\usepackage{exscale}
\usepackage{relsize}
\usepackage{amssymb}
\usepackage{hyperref}
\usepackage{url}
\allowdisplaybreaks

\hypersetup{hypertex=green,
            colorlinks=true,
            linkcolor=green,
            anchorcolor=green,
            citecolor=red}

\theoremstyle{definition}
\newtheorem{theorem}{Theorem}[section]
\newtheorem{definition}{Definition}[section]

\newtheorem{remark}{Remark}[section]
\newtheorem{lemma}{Lemma}[section]

\numberwithin{equation}{section}%
\numberwithin{table}{section}%
\numberwithin{figure}{section}

\def\3bar{{|\hspace{-.02in}|\hspace{-.02in}|}}
\def\td{\text{div}}
\def\tc{\text{curl}}
\def\d{\text{d}}

\begin{document}
\title[]{Curl-curl conforming elements on tetrahedra}
\author{Qian Zhang}
\email{go9563@wayne.edu}
\address{Department of Mathematics, Wayne State University, Detroit, MI 48202, USA. }

\author{Zhimin Zhang}
\email{ zmzhang@csrc.ac.cn;  zzhang@math.wayne.edu}
\address{Beijing Computational Science Research Center, Beijing 100193, China; and Department of Mathematics, Wayne State University, Detroit, MI 48202, USA.}
\thanks{The research of this work is supported in part by the National Natural Science Foundation of China grants NSFC 11871092 and NASF U1930402.}


\date{May 31, 2020}
\begin{abstract} 
In \cite{curlcurl-conforming2D}, we proposed a family of $H(\text{curl}^2)$-conforming elements on both a triangle and a rectangle. The elements providesa brand new method to solve the quad-curl problem in 2 dimensions. In this paper, we turn our focus to 3 dimensions and construct $H(\text{curl}^2)$-conforming finite elements on tetrahedra. The newly proposed elements have been proved to have the optimal   interpolation error estimate. Having the tetrahedral elements, we can solve the  quad-curl problem in any Lipschitz domain by the conforming finite element method. We also provide several numerical examples of using our elements to solve the quad-curl problem. The results of the numerical experiments show the correctness of our elements.
\end{abstract}	
\maketitle
\section{Introduction}
The quad-curl problem  are involved in  various practical problems, such as inverse electromagnetic scattering theory \cite{Cakoni2017A,Monk2012Finite, Sun2016A} or magnetohydrodynamics \cite{Zheng2011A}.
As its name implies, this problem involves a fourth-order curl operator which makes it much more challenging to solve than the lower-order electromagnetic problem \cite{Monk2003,Peter1992A,Teixeira2008Time,Jin1993Finite,Daveau2009A,Li2013Time,Monk1994Superconvergence}.
The regularity of this problem was studied by Nicaise \cite{Nicaise2016}, Zhang \cite{Zhang2018M2NA}, and Chen et al. \cite{CQXDG}. As for the numerical methods,
Zheng et al. developed a nonconforming finite element method for this problem in \cite{Zheng2011A}. This method has low computational cost since it has small number of degrees of freedom (DOFs), but it bears the disadvantage of low accuracy. Based on N\'ed\'elec elements, a discontinuous Galerkin method and a weak Galerkin method were presented in \cite{Qingguo2012A} and \cite{quadcurlWG}, respectively. In addition,  error estimates for discontinous Galerkin methods based on a relatively low regularity assumption of the exact solution are proposed in \cite{CCXDG,CQXDG}.
Another approach to deal with the quad-curl operator is to introduce an auxiliary variable and reduce the original problem to a second-order system \cite{Sun2016A}. Zhang proposed a different mixed scheme \cite{Zhang2018M2NA}, which relaxes the regularity requirement in theoretical analysis.

However, the most natural way to solve this problem is the conforming finite element method. In \cite{curlcurl-conforming2D}, the authors and another collaborator constructed curl-curl-conforming or $H(\text{curl}^2)$-conforming elements in 2 dimensions (2D) to solve the quad-curl problem. In three dimensions (3D), the numerical solution derived by the existing $H^{2}$-conforming (or $C^{1}$-conforming) elements ($\bm u\in \bm L^2$ and $\nabla\times\bm u\in \bm H^1$) \cite{zhang2009family} converges to an $H^2$ projection of the exact solution.
The distance between such a projection and the exact solution may be a positive constant since $C^{\infty}_0$ may not be dense in $H(\tc^2)\cap H(\td^0)$ under a specific norm. Indeed, the treatment of boundary conditions is also an issue when using $H^2$-conforming elements to solve the quad-curl problem.
Also, Neilan constructed a family of $\bm H^1(\tc)$-conforming elements ($\bm u\in \bm H^1$ and $\nabla\times\bm u\in \bm H^1$) in \cite{neilan2015discrete} (see \cite{falk2013stokes} for the 2D case).
The family of  elements can also lead to conforming approximations of the quad-curl problem. However, in this paper, we derive a conforming finite element space for $H(\tc^2;\Omega)$ ($\bm u\in \bm L^2$ and $\nabla\times\bm u\in \bm H^1$) where the function regularity is weaker than the space $\bm H^1(\tc)$. Such types of elements, to the best of the authors' knowledge,  are not available in the literature. Due to the large kernel space and the natural divergence-free property of the curl operator $\nabla\times$, the construction of $H(\text{curl}^2)$-conforming elements is more difficult than 2D case.

Our paper starts by describing the tetrahedral curl-curl-conforming finite elements. The unisolvence and conformity of our $H(\text{curl}^2)$-conforming finite elements can be verified by a rigorous mathematical analysis. Moreover, our new elements have been proved to possess  good interpolation properties. Although  the involvement of normal derivatives to edges render the DOFs on a general element failing to relate to those on the reference element, we constructed intermediate elements whose DOFs can be related to those on the reference element and are close to our elements.
In this way, we prove the optimal error estimate of the finite element interpolation.
    In our construction, the number of the degrees of freedom for the lowest-order element is 315. Because of the large number of DOFs, it's hard the compute the Lagrange-type basis functions by the traditional method. Hence we employ the method proposed in \cite{Dom2008Algorithm} to obtain the basis functions on a general element.
	
The rest of the paper is organized as follows. In section 2 we list some function spaces and notations. Section 3 is the technical part, where we construct the $H(\text{curl}^2)$- conforming finite elements on a tetrahedron. In section 4 we give the error estimate for the interpolation. In section 5 we use our newly proposed elements to solve the quad-curl problem and give some numerical results to verify the correctness of our method. Finally, some concluding remarks and possible future works are given in section 6. We present how to implement the finite elements in Appendix and provide the code for it.  

\section{Preliminaries}
Let $\Omega\in\mathbb{R}^3$ be a simply-connected Lipschitz domain. We adopt standard notations for Sobolev spaces such as $W^{m,p}(D)$ or $W^{m,p}_0(D)$ on a simply-connected sub-domain $D\subset\Omega$ equipped with the norm $\left\|\cdot\right\|_{m,p,D}$ and the semi-norm $\left|\cdot\right|_{m,p,D}$. If $p=2$, the space $W^{m,2}(D)$ is  exactly the space $H^{m}(D)$ with the norm $\left\|\cdot\right\|_{m,D}$. If $m=0$,  the space $W^{0,p} (D)$ coincides with $ L^p(D)$. When $D=\Omega$, we drop the subscript $D$ for ease of notation. We use  $\bm W^{m,p}(D)$, $\bm H^{m}(D)$,   and ${\bm L}^p(D)$ to denote the vector-valued Sobolev spaces $\left(W^{m,p}(D)\right)^3$, $\left(H^{m}(D)\right)^3$, and $\left(L^p(D)\right)^3$.

Let ${\bm u}=(u_1, u_2,u_3)^T$ and ${\bm w}=(w_1, w_2, w_3)^T$, where the superscript $T$ denotes the transpose,
then ${\bm u} \times {\bm w} = (u_2w_3-w_2u_3, w_1u_3-u_1w_3, u_1w_2-w_1u_2)^T$ and 
$\nabla \times {\bm u} = (\partial_{x_2}u_{3}- \partial_{x_3}u_{2}, \partial_{x_3}u_{1} -  \partial_{x_1}u_{3}, \partial_{x_1}u_{2} - \partial_{x_2}u_{1} )^T$.
For convenience, here and hereinafter we abbreviate the partial differential operators $\frac{\partial }{\partial x_i}$
to $\partial_{ x_i}$.
We denote $(\nabla\times)^2\bm u=\nabla\times\nabla\times\bm u$.

For $s=1,2$, we define
\begin{align*}
	H(\text{curl}^s;D)&:=\{\bm u \in {\bm L}^2(D):\; (\nabla \times)^j \bm u \in \bm L^2(D), \; j=1,s\}
\end{align*}
with scalar products and norms are defined by
	\[(\bm u,\bm v)_{H(\tc^s;D)}=(\bm u,\bm v)+\sum_{j=1}^s((\nabla\times)^j\bm u,(\nabla\times)^j \bm v)\]
	and
	\[\left\|\bm u\right\|_{H(\tc^s;D)}=\sqrt{(\bm u,\bm u)_{H(\tc^s;D)}}.\]
The spaces $H_0(\text{curl}^s;D)(s=1,\;2)$ with vanishing boundary conditions are defined as 
\begin{align*}
	&H_0(\text{curl};D):=\{\bm u \in H(\text{curl};D):\;{\bm n}\times\bm u=0\; \text{on}\ \partial D\},\\
	&H_0(\text{curl}^2;D):=\{\bm u \in H(\text{curl}^2;D):\;{\bm n}\times\bm u=0\; \text{and}\; \nabla\times \bm u=0\;\; \text{on}\ \partial D\}.
\end{align*}

For a subdomain $D$, a face $f$, or an edge $e$, we use $P_k$ and $\widetilde P_k$ to represent the space of polynomials on them with degree of no more than $k$ and the space of homogeneous polynomials, respectively.
Denote $\bm P_k=\left(P_k(D)\right)^3$ and $\widetilde{\bm P}_k=\big(\widetilde{P}_k(D)\big)^3$. 
We also denote
\begin{align*}
&\mathcal R_k=\bm P_{k-1}\oplus \mathcal S_k \text{\ with\ }\mathcal S_k=\{{\bm p}\in \widetilde{\bm P}_k\big| \ \bm x\cdot \bm p=0\},
\end{align*}
whose dimension is
\begin{align*}
&\dim{\mathcal R_k}=\frac{k(k+2)(k+3)}{2}.
\end{align*}
Note that, for the space $\bm P_k$, we have the following decomposition \cite{da2018lowest}
\begin{align}
&\bm P_k=\nabla P_{k+1} \oplus \bm x\times \bm P_{k-1},\label{Pkdecomp1}\\
&\bm P_k=\nabla\times \mathcal R_{k+1} \oplus  \bm x P_{k-1}\label{Pkdecomp2}.
\end{align}
The dimension of $ \bm x\times \bm P_{k-1}$ is $\dim{\bm P_k}-\dim{P_{k+1}}+1$ and the dimension of $\nabla\times \mathcal R_{k+1}$ is $\dim{\bm P_k}-\dim{P_{k-1}}$.

We adopt the following curl preserving Piola mapping \cite{piolamapping} to relate a function $\bm u$ on a general element $K$ to a function $\hat{\bm u}$ on the reference element $\hat K$ (the tetrahedron with vertices (0,0,0), (1,0,0), (0,1,0), and (0,0,1)):
\begin{align}
\bm u \circ F_K = B_K^{-T} \hat{\bm u},\label{mapping-u}
\end{align}
where $B_K$ is a 3-by-3 matrix from the affine mapping
\begin{align}\label{mapping-domain}
F_K(\hat{\bm x})= B_K\hat{\bm x}+ \bm b_K.
\end{align}
By a simple computation, we have
\begin{align}
(\nabla\times\bm u) \circ F_K &= \frac{B_K}{\det(B_K)} \hat{\nabla}\times\hat{ \bm u},\label{mapping-curlu}
\end{align}
and 
\begin{align}
 \bm n\circ F_K&= \frac{B_K^{-T} \hat {\bm n}}{\left\|B_K^{-T} \hat {\bm n}\right\|}\label{n},\\
\bm \tau\circ F_K&= \frac{B_K \hat {\bm \tau}}{\left\|B_K \hat {\bm \tau}\right\|}\label{tau},
\end{align}
for the unit normal vector $\bm n$ and the unit tangential vector $\bm{\tau}$ to $\partial K$.
%

\section{The Finite Elements on a Tetrahedron}

In this section, we will construct a family of finite elements, built on a tetrahedron, which is conforming in the space $H(\tc^2)$. To this end, we first introduce the following lemma which tells us the continuity conditions the finite elements should satisfy.

\begin{lemma}\label{rqmt1}
	Let $K_1$ and $K_2$ be two non-overlapping Lipschitz domains having a common face $\Lambda$ such that $\overline{{K}_1}\cap\overline{{K}_2} = \Lambda$. Assume that ${\bm u}_1 \in H(\text{curl}^2;{K}_1)$,
	${\bm u}_2 \in H(\text{curl}^2;{K}_2)$, and $\bm u \in \bm L^2({K}_1 \cup {K}_2 \cup \Lambda)$ is defined by
	\begin{equation*}
		\displaystyle{\bm u}=
		\begin{cases}
			&\bm u_1,\quad  \text{in}\ {K}_1,\\[0.2cm]
			&\bm u_2,\quad  \text{in}\ {K}_2.
		\end{cases}
	\end{equation*}
	Then $\bm u_1 \times \bm n_1 = -\bm u_2 \times \bm n_2$\ and $\nabla\times \bm u_1 \times \bm n_1 =-\nabla\times \bm u_2\times \bm n_2$ on $\Lambda$ implies that $\bm u \in H(\text{curl}^2;\text{K}_1 \cup \text{K}_2 \cup \Lambda)$, where $\bm n_i$ ($i=1,2$) is the unit outward normal vector to $\partial K_i$ and note that  $\bm n_1 = - \bm n_2$.
\end{lemma}
\begin{proof}
	The proof is similar with that of Lemma 3.1 in \cite{curlcurl-conforming2D}.
\end{proof}


From Lemma \ref{rqmt1}, we know that the elements we will construct should satisfy the following continuity conditions:
\begin{itemize}
	\item $\bm u_1 \times \bm n_1 = -\bm u_2 \times \bm n_2$.
	\item $\nabla \times \bm u_1\times \bm n_1=-\nabla \times \bm u_2\times \bm n_2.$
		\item $\nabla \times \bm u_1\cdot \bm n_1=\nabla \cdot (\bm u_1\times \bm n_1)= -\nabla \cdot (\bm u_2\times \bm n_2)=-\nabla \times \bm u_2\cdot \bm n_2.$
	\end{itemize}
\begin{figure}
	\includegraphics[scale=0.5]{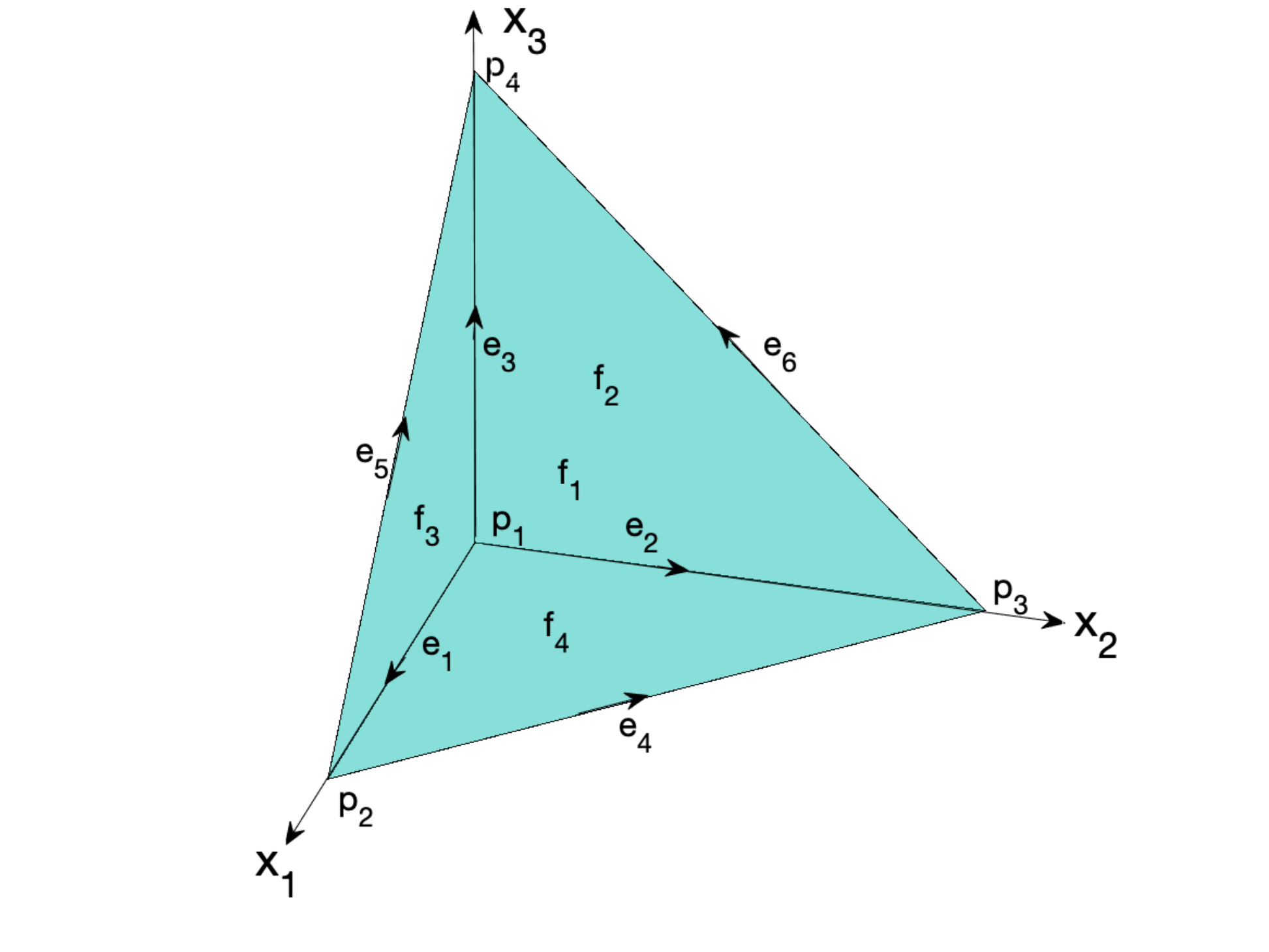}
	\caption{A reference tetrahedron}\label{reftetra}
\end{figure}
The last two conditions imply $\nabla\times\bm u_1$ and $\nabla\times\bm u_2$ are continuous across the face.
Based on above continuity conditions, we give the definition of the elements as follows.
\begin{definition}[Curl-curl-conforming elements or ${H}(\mathrm{curl}^2)$-conforming elements on a tetrahedron] \label{tetra-DOF}
For any integer $k \geq 7$, the ${H}(\mathrm{curl}^2)$-conforming elements are defined by the triple:
	\begin{equation*}
	\begin{split}
	& {K}\;\text{is a tetrahedron}\ (see \ Figure\  \ref{reftetra})\\
	&P_{ {K}} = \mathcal{R}_k,\\
	&\Sigma_{ {K}} = \bm M_{ {\bm p}}(  {\bm u}) \cup \bm M_{ {\bm e}}(  {\bm u}) \cup \bm M_{ {f}}(  {\bm u}) \cup \bm M_{ {K}}(  {\bm u}),
	\end{split}
	\end{equation*}
where $\Sigma_{ {K}}$ is the set of the DOFs defined by the following.
	\begin{itemize}
		\item Vertex DOFs: $M_{  {p}}(  {\bm u})$  
		\begin{itemize}
		\item $ {\nabla}\times {\bm u}(  {\bm p}_i), \ i=1,2,3,4;$
        \item $D( {\nabla}\times {\bm u})(  {\bm p}_i), \ i=1,2,3,4,$  except $\partial_{x_3} ( {\nabla}\times {\bm u})_{3 }(  {\bm p}_i)$;        \item $D^2( {\nabla}\times {\bm u})(  {\bm p}_i), \ i=1,2,3,4,$ except $\partial^2_{x_1   x_1}( {\nabla}\times {\bm u})_1(  {\bm p}_i)$, $\partial^2_{x_2   x_2}( {\nabla}\times {\bm u})_{2}(  {\bm p}_i)$, $\partial^2_{x_3   x_3}( {\nabla}\times {\bm u})_{3}(  {\bm p}_i)$.
       \end{itemize}	
   Here, we use $D\bm v$ and $D^2\bm v$ to represent all the first-order and second-order derivatives of $\bm v$.
       \item Edge DOFs:  $M_{  e}(  {\bm u})$  
       \begin{itemize}
       	\item $\int_{  {\bm e}_i} {\bm u} \cdot {\bm{\tau}_i}  q\mathrm{d} s, \ \forall  q\in P_{k-1}(  {e}_i),\ i=1,2,\cdots,6;$
       	\item $( {\nabla}\times {\bm u})(  {\bm e}^0_{ij}), \ i=1,2,\cdots,6,\ j=1,2,\cdots,k-6;$
       	\item  $\nabla( {\nabla}\times {\bm u}\cdot \bm v_i)(  {\bm e}^1_{ij})\cdot {\bm n}_i , \ i=1,2,\cdots,6,\ j=1,2,\cdots,k-5,\ \bm v_i= {\bm \tau}_i,  {\bm n}_i, \text{or} \  {\bm m}_i;$
       	\item  $\nabla( {\nabla}\times {\bm u}\cdot \bm v_i)(  {\bm e}^1_{ij})\cdot {\bm m}_i , \ i=1,2,\cdots,6,\ j=1,2,\cdots,k-5,\ \bm v_i= {\bm \tau}_i\ \text{or}\  {\bm n}_i,$
       \end{itemize}	
   	where $ {\bm \tau}_i,  {\bm n}_i,  {\bm m}_i$ are the unit tangential vector and two unit normal vectors to the edge $ \bm e_i$, and ${\bm e}^0_{ij}$ (or ${\bm e}^1_{ij}$) are $k-6$ (or $k-5$) distinct nodes on edge $\bm {e}_i$.
   \item Face DOFs:  $M_{f}({\bm u})$   
   \begin{itemize}
   \item $\frac{1}{area(f_i)}\int_{  f_i} {\bm u} \cdot  {\bm q}\mathrm{d} S, \ \forall     {\bm q}\circ F_K=B_K\hat{\bm q},\   \hat{\bm q}\in P_{k-3}( \hat f_i)[\hat{\bm x}-(\hat{\bm x}\cdot\hat{\bm \nu}_i)\hat{\bm \nu}_i]|_{\hat f_i},\ i=1,2,3,4;$   	 	
   \item $\frac{1}{area(f_i)}\int_{  f_i} {\nabla}\times {\bm u} \cdot  {\bm q}\mathrm{d} S, \ \forall  {\bm q}\in P_{k-7}(f_i)\bm{t}_{i,1}\oplus P_{k-7}(f_i)\bm{t}_{i,2},\ i=1,2,3,4;$
   	\item $\int_{  f_i} {\nabla}\times {\bm u} \cdot  {\bm n}_i  { q} \mathrm{d} S, \ \forall  {q}\in P_{k-7}(  f_i)/\mathbb{R},\ i=1,2,3,4,$
   \end{itemize}
where $B_K$ is from the affine mapping $F_K$ which maps $\hat K$ to $K$, $\hat{\bm \nu}_{i}$ is the unit outward normal vector to the reference face ${\hat f}_i$ associated with ${f}_i$, and $\bm t_{i,j},\ j=1,2$ are the two noncollinear vectors on the face ${f}_i$.
   	 \item Interior DOFs:  $M_{K}({\bm u})$   
   	\begin{itemize}
\item $\int_{  K} {\bm u} \cdot  {\bm q}\mathrm{d}  V, \ \forall {\bm q}\circ F_K= \det(B_K)^{-1}B_K\hat{\bm q},\ \hat{\bm q}\in P_{k-4}(  \hat K) \hat   {\bm x}$.
\item $\int_{  K} \nabla\times{\bm u} \cdot  {\bm q}\mathrm{d}  V, \ \forall {\bm q}\circ F_K= \det(B_K)B_K^{-T}\hat{\bm q},\ \hat{\bm q}\in \hat{\bm x}\times \bm P_{k-7}(\hat K)$.
   \end{itemize}	
\end{itemize}	
\end{definition}

\begin{remark}	
We exclude the DOFs $\partial_{x_3} ( {\nabla}\times {\bm u})_{3 }(  {\bm p}_i)$, $\partial^2_{x_1   x_1}( {\nabla}\times {\bm u})_1(  {\bm p}_i)$, $\partial^2_{x_2   x_2}( {\nabla}\times {\bm u})_{2}(  {\bm p}_i)$, $\partial^2_{x_3   x_3}( {\nabla}\times {\bm u})_{3}(  {\bm p}_i)$, $\nabla( {\nabla}\times {\bm u}\cdot \bm m_i)(  {\bm e}^1_{ij})\cdot {\bm m}_i$, and  $\int_{  f_i} {\nabla}\times {\bm u} \cdot  {\bm n}_i  \mathrm{d} S$ because of the divergence-free property of $\nabla\times\bm u$.
\end{remark}
\begin{remark}	
The DOFs for $\nabla\times\bm u$ are the same as the elements in \cite{neilan2015discrete}.
\end{remark}

Now we have $ (26\times 4)$ node DOFs,  $(6k+18(k-6)+30(k-5))$ edge DOFs,  $(2(k-2)(k-1)+6(k-6)(k-5)-4)$ face DOFs,  and  $((k-3)(k-2)(k-1)/6+(k-5)(k-4)(k-3)/2-(k-4)(k-3)(k-2)/6+1)$ interior DOFs, and therefore
\begin{align*}
\dim(P_{ {K}})&=26\times 4+6k+18(k-6)+30(k-5)+2(k-2)(k-1)+6(k-6)(k-5)-4\\
&+k^3/2 - 11k^2/2 + 21k - 26=\frac{k(k+2)(k+3)}{2}=\dim \mathcal{R}_k.
\end{align*}
Since $k\geq 7$, the minimum number of DOFs  is 315.

\begin{theorem}
	The DOFs defined in Definition \ref{tetra-DOF} are well-defined for any $\bm u\in \bm H^{1/2+\delta}(K)$  with $\delta>0$ and $\nabla\times\bm u\in \bm C^2(\bar K)$.
\end{theorem}
\begin{proof}
	Since $\nabla\times\bm u \in \bm C^2(\bar K)$, the vertex DOFs  and edge DOFs involving $\nabla\times\bm u$ are well-defined.  It follows from the Cauchy-Schwarz inequality that the face DOFs and interior DOFs are well-defined since $\bm u, \nabla\times\bm u\in \bm H^{1/2+\delta}(K)$ and  $\bm u|_{\partial K},(\nabla\times\bm u)|_{\partial K}\in \bm H^{\delta}(\partial K)$. By the argument in the proof of Lemma 5.38 in \cite{Monk2003}, the DOF $\int_{\bm  e_i}\bm u\cdot \bm\tau_i q\d s$ is well-defined if $\bm u\in \bm H^{1/2+\delta}(K)$ and $\nabla\times\bm u\in \bm L^p(K)$ with $p>2$. We has completed the proof since $\nabla\times\bm u\in \bm C^2(\bar K)\subset \bm L^p(K).$
\end{proof}
\begin{figure}
	\includegraphics[scale=0.35]{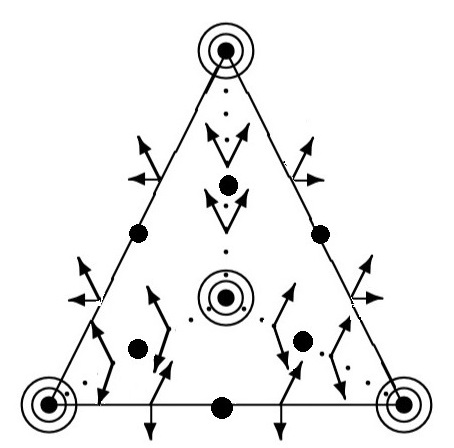}
	\includegraphics[scale=0.5]{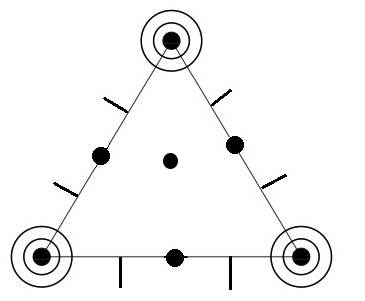}
	\caption{The vertex and edge DOFs of a 3D Argyris element and a 2D Argyris element}\label{argyris}
\end{figure}
\begin{theorem}
	The finite elements given by Definition \ref{tetra-DOF} are unisolvent and conforming in $H(\mathrm{curl}^2)$.
\end{theorem}
\begin{proof}
	(\romannumeral 1). To prove the $H(\text{curl}^2)$ conformity, it suffices to prove $ {\bm u} \times  {\bm \nu}_i = 0$ 
	and $ {\nabla}\times  {\bm u}=0$ on a face when all DOFs associated with this face vanish.  
	Without loss of generality, we only consider the face $  f_4$ in Figure \ref{reftetra}. 
	By the first kind of edge DOFs, $$ {\bm u}\cdot  {\bm \tau}_i =0 \text{\ on each edge} \  \bm{ e}_i.$$ 
	Furthermore, the node DOFs and edge DOFs involving $ {\nabla}\times  {\bm u}$  perform like the 3D Argyris element (if restricted on a face, it will be  the 2D Argyris element, see Fig \ref{argyris}),  we can get 
	\[ {\nabla}\times  {\bm u}=0 \  \text{and} \ D( {\nabla}\times  {\bm u})=0 \ \text{on each edge}.\]
	Restricted on the face $f_4$, we have 
	\[ {\nabla}\times  {\bm u}\cdot {\bm \nu}_4=0 \  \text{and} \ D( {\nabla}\times  {\bm u}\cdot {\bm \nu}_4)=0 \ \text{on each edge of this face}.\]
	Hence we can rewrite it as,
	\[ {\nabla}\times  {\bm u}\cdot {\bm \nu}_4=  (\lambda_1^2  \lambda_2^2 \lambda_3^2)|_{f_4}\varphi, \ \varphi\in P_{k-7}(  f_4),\]
	where $\lambda_i, i=1,2,3$ are the barycentric coordinates such that $\lambda_i(\bm p_j)=\delta_{ij}$ for $j=1,2,3,4$. 
	By the integration by parts and $ {\bm u}\cdot  {\bm \tau}_i |_{  \bm e_i}=0$,
	\begin{equation*}
	\int_{ {f}_4} {\nabla}\times  {\bm u}\cdot \bm \nu_4 C\d S=\int_{\partial  {f}_4} { C} {\bm u}\cdot  {\bm{\tau}_i}\d s+	\int_{ {f}_4}{\bm u}\times\bm \nu_4\cdot\nabla C\d S=0,
	\end{equation*}
	which together with the third kind of face DOFs, we get
	\[ {\nabla}\times  {\bm u}\cdot \bm \nu_4=0.\]
	Hence there exists a $\phi\in P_k(f_4)$ s.t., restricted on the face $f_4$,
	\begin{align*}
	 {\bm u}\cdot {\bm t_i}=  {\nabla}\phi\cdot {\bm t_i}, i=1,2, 
	 	 \end{align*}
    where $\bm t_i,i=1,2$ are two mutually orthogonal unit vectors on the face $f_4$. 
    In fact, since $0=({\nabla}\times  {\bm u}\cdot \bm \nu_4)|_{f_4}=\nabla_f\cdot(\bm u\times \bm \nu_4)|_{f_4}$, we have $\bm u\times \bm \nu_4|_{f_4}=\nabla_f\times \phi=(\nabla\phi\cdot \bm t_2)\bm t_1+(-\nabla\phi\cdot \bm t_1)\bm t_2$ which implies ${\bm u}\cdot {\bm t_i}=  {\nabla}\phi\cdot {\bm t_i}$ since $\bm u\times \bm \nu_4|_{f_4}=-(\bm u\cdot \bm t_1)|_{f_4}\bm t_2+(\bm u\cdot \bm t_2)|_{f_4}\bm t_1$.

    According to $ {\bm u}\cdot  {\bm \tau}_i |_{  \bm e_i}=0$, we have $  \nabla{\phi}\cdot {\bm {\tau}}_i=0$ on $\partial   f_4$, which implies that ${\phi}$ can be chosen as
	\[{\phi}= (\lambda_1  \lambda_2 \lambda_3)|_{f_4}{\psi} \text{ for some }\ {\psi} \in P_{k-3}(  f_4).\]
	Applying integration by parts, we obtain, for a $\bm q$ such that $\bm q\cdot\bm \nu_4=0$,
	\[  ({\bm u},\bm q)_{f_4}=(\phi,\nabla\cdot\bm q)_{f_4}=((\lambda_1  \lambda_2 \lambda_3)|_{f_4}{\psi},\nabla\cdot\bm q)_{f_4}.\]
	By the first kind of face DOFs, we can pick a $\bm q$ s.t. $\nabla\cdot\bm q=\psi$ and $q\cdot\bm \nu_4=0$, and hence, arrive at $\psi=0$, i.e. $\bm u\times \bm \nu_4=0.$
	Recall that $D( {\nabla}\times  {\bm u}\times   {\bm \nu_4})=0$ and $ {\nabla}\times  {\bm u}\times   {\bm \nu_4}=0$ on each edge of $  f_4$, i.e., 
	\begin{align*}
	( {\nabla}\times  {\bm u})\cdot \bm t_j =    (\lambda_1^2  \lambda_2^2 \lambda_3^2)|_{f_4}\varphi_j,\ \varphi_j\in P_{k-7}(  f_4),j=1,2.
	\end{align*}
	The second kind of face DOFs yields $ {\nabla}\times  {\bm u}\times   \bm \nu_4=0$ which together with $ {\nabla}\times  {\bm u}\cdot   \bm \nu_4=0$ leads to $ {\nabla}\times  {\bm u}=0$ on the face $  f_4$.
	
	(\romannumeral 2).   Now, we consider the unisolvence.
	We only need to prove that vanishing all DOFs for $ \hat{\bm u }\in P_{\hat {K}}$ yields $ \hat{\bm u}=0$ on  the reference element $\hat K$. For ease of notation, we will omit the hat notation. By virtue of the fact that $ {\nabla}\times  {\bm u} = 0$ on $\partial   K$, we can rewrite $ {\nabla}\times  {\bm u}$ as:
	\[ {\nabla}\times  {\bm u} =  x_1   x_2   x_3 (1-  x_1 -  x_2-  x_3 ){\bm \Phi} \text{ with }\ {\bm  \Phi}=(\bm  \Phi_1,\bm  \Phi_2,\bm  \Phi_3)^T\in \bm P_{k-5}(  K),\]
	and hence
	\begin{align*}
	\partial_{x_1}( {\nabla}\times  {\bm u} )_{1}	=   x_2  x_3 [(1-2  x _1-  x_2-  x_3 )\bm \Phi_1+(1-  x_1 -  x_2-  x_3 ) x_1 \partial_{ {x_1}}\bm \Phi_1],\\
	\partial_{x_2}( {\nabla}\times  {\bm u} )_{2}	=   x_1  x_3 [(1-2  x_2 -  x_1-  x_3 )\bm \Phi_2+(1-  x_1 -  x_2-  x_3 ) x_2 \partial_{ {x_2}}\bm \Phi_2],\\
	\partial_{x_3}( {\nabla}\times  {\bm u} )_{3}	=   x_1  x_2 [(1-2  x_3 -  x_1-  x_2 )\bm \Phi_3+(1-  x_1 -  x_2-  x_3 ) x_3 \partial_{ {x_3}}\bm \Phi_3].
	\end{align*}
	When  $  x_1=0$, $\partial_{x_2}( {\nabla}\times  {\bm u} )_{2}$+$\partial_{ x_3}( {\nabla}\times  {\bm u} )_{3}=0$ which leads to
	$\partial_{ {x_1}}( {\nabla}\times  {\bm u} )_{1}	=0$ because $\nabla\cdot\nabla\times\bm u=\partial_{ {x_1}}( {\nabla}\times  {\bm u} )_{1}+\partial_{x_2}( {\nabla}\times  {\bm u} )_{2}+\partial_{ x_3}( {\nabla}\times  {\bm u} )_{3}=0$. It implies  $\bm \Phi_1$ has a factor $x_1$.  
	Similarly, $\bm \Phi_2$  has a factor $  x_2$ and $\bm \Phi_3$ has a factor $x_3$.  
	Then 
	\[ {\nabla}\times  {\bm u} =  x_1   x_2   x_3 (1-  x_1 -  x_2-  x_3 )[  x_1 \tilde{\bm \Phi}_1,   x_2\tilde{\bm \Phi}_2,   x_3 \tilde{\bm \Phi}_3]^T\]
with $\tilde{\bm  \Phi}=[\tilde{\bm  \Phi}_1,\tilde{\bm  \Phi}_2,\tilde{\bm  \Phi}_3]\in \bm P_{k-6}( K),i=1,2,3.$ 
	Due to the second kind of  interior vanishing DOFs and the fact that
	\[ ({\nabla}\times  {\bm u},\nabla q )_K=\langle   {\bm u}\times   {\bm n}, \nabla q  \rangle_{\partial   K}=0, \ \forall q \in P_{k-5}(K),\]
	we have
	\begin{align}
	( {\nabla}\times  {\bm u}, \bm q )_{  K}=0, \ \forall \bm q \in \bm P_{k-6}(K),\label{curlu-interior}
	\end{align}
	here we used the decomposition \eqref{Pkdecomp1}. By setting $\bm q=\tilde{\bm  \Phi}$ in \eqref{curlu-interior}, we get
	\[ {\nabla}\times  {\bm u}=0\text{ in }  K.\] 
	Therefore, we can choose a $\Psi=   x_1   x_2   x_3 (1-  x_1 -  x_2-  x_3 )\tilde{\Psi}$ with $\tilde{\Psi}\in P_{k-4}(  K)$ such that
	\[ {\bm u}= {\nabla} \Psi.\]
	Again, by applying  integration by parts, 
	\[(  {\bm u}, {\bm  q} )_{  K}=( {\nabla} \Psi , {\bm  q} )_{  K}= (\Psi ,  {\nabla} \cdot {\bm  q} )_{  K}=(  x_1   x_2   x_3 (1-  x_1 -  x_2-  x_3 )\tilde{\Psi},  {\nabla} \cdot {\bm  q} )_{  K}.\]
	Using the first kind of interior DOFs and choosing a ${\bm  q}$ s.t. $ {\nabla} \cdot {\bm  q} = \tilde{\Psi}$, we get $\tilde{\Psi}=0$ and hence $ {\bm u}=0$.
\end{proof}

\section{Error Estimate of Finite Element Interpolation}
Provided $\bm u \in \bm H^{1/2+\delta}(K)$ with $\delta >0$ and $ \nabla \times \bm u \in \bm C^2(\bar K)$,  we can
define an $H(\tc^2)$ interpolation operator on $K$ denoted as $\Pi_K$ by
\begin{align}
	 M_p(\bm u-\Pi_K\bm u)=0,\  M_e(\bm u-\Pi_K\bm u)=0,\  M_f(\bm u-\Pi_K\bm u)=0,\ \text{and}\ M_K(\bm u-\Pi_K\bm u)=0,
\end{align}\label{def-interp}
where $M_p,\  M_e$, $M_f$, and $M_K$ are the sets of DOFs in Definition \ref{tetra-DOF}.

The DOFs of the  finite elements  defined in Definition \ref{tetra-DOF} involve  normal derivatives to edges, we can not relate the interpolation $\Pi_K$ on a general element $K$ 
to $\Pi_{\hat  K}$ on the reference element $\hat K$ by the mapping \eqref{mapping-u}. 
To estimate the interpolation error, we introduce finite elements slightly different from our elements, but the corresponding interpolation on $K$ and that on $\hat K$ can be related via the mapping \eqref{mapping-u}.

\begin{definition} \label{tetra-affine}
	For any integer $k \geq 7$, ${H}(\mathrm{curl})$-conforming elements are defined by the triple:
	\begin{equation*}
	\begin{split}
	& {K}\;\text{is a tetrahedron}\ (see \ Figure\  \ref{reftetra})\\
	&P_{ {K}} = \mathcal{R}_k,\\
	&\tilde \Sigma_{ {K}} =  \tilde M_{ { p}}(  {\bm u}) \cup \tilde M_{ { e}}(  {\bm u}) \cup \tilde M_{ {f}}(  {\bm u}) \cup \tilde M_{ {K}}(  {\bm u}),
	\end{split}
	\end{equation*}
	where $\tilde\Sigma_{ {K}}$ are the DOFs obtained by the following slight changes to $\Sigma_{ {K}}$ in Definition \ref{tetra-DOF} with the other DOFs staying the same.
	\begin{itemize}
			\item Edge DOFs:  $\tilde M_{  e}(  {\bm u})$  
		\begin{itemize}
		\item  $\big(\nabla( {\nabla}\times {\bm u}\cdot (\bm v^i_m\times\bm v^i_n))\cdot {\bm v}^i_l \big)(  {\bm e}^1_{ij}), \ i=1,2,\cdots,6,\ j=1,2,\cdots,k-5$, 
			 $\{l,[m,n]\}=\{2,[1,2]\}, \ \{2,[1,3]\}, \ \{2,[2,3]\},\  \{3,[1,2]\},\text{ or } \{3,[1,3]\},$ 
		\end{itemize}	
		where ${\bm v}^i_n\ (n=1,2,3)$ are the three edges  intersected at the initial point of $\bm  e_i$ with ${\bm v}^i_1=\bm  e_i$ and  $  {\bm e}^1_{ij}\ ( j=1,2,\cdots,k-5)$ are defined as in Definition \ref{tetra-DOF}.
		\end{itemize}	
\end{definition}
\begin{remark}
	With these small changes, the elements are conforming in $H(\tc)$ rather than	$H(\tc^2)$.
\end{remark}

Provided $\bm u \in \bm H^{1/2+\delta}(K)$  with $\delta >0$ and $ \nabla \times \bm u \in \bm C^2(\bar K)$,  we can
define an $H(\tc)$ interpolation operator on $K$ denoted as $\Lambda_K$ by
\begin{align}\label{def-interp-aff}
\tilde M_p(\bm u-\Lambda_K\bm u)=0,\ \tilde M_e(\bm u-\Lambda_K\bm u)=0,\ \tilde M_f(\bm u-\Lambda_K\bm u)=0,\ \text{and}\ \tilde M_K(\bm u-\Lambda_K\bm u)=0,
\end{align}
where $\tilde M_p,\ \tilde M_e$, $\tilde M_f$, and $\tilde M_K$ are the sets of DOFs in Definition \ref{tetra-affine}.

\begin{lemma}\label{relation}
Assume that $\Lambda_K$ is well-defined. Then under the transformation \eqref{mapping-u}, we have
$\Lambda_K \bm u\circ F_K=B_K^{-T}\Lambda_{\hat K}\hat {\bm u}$.
\end{lemma}
\begin{proof}
With the slight changes, the new DOFs are identical with those in $\tilde \Sigma_{\hat K}$ under \eqref{mapping-u}
\[\nabla\big( {\nabla}\times {\bm u}\cdot (\bm v^i_m\times\bm v^i_n)\big)\cdot {\bm v}^i_l= B_K^{-T}\hat \nabla \big({B_K\hat{\nabla}\times \hat{\bm u}}\cdot B_K^{-T}(\hat{\bm v}^i_m\times\hat{\bm v}^i_n)\big)\cdot B_K\hat{\bm v}^i_l =\hat \nabla \big({\hat{\nabla}\times \hat{\bm u}}\cdot (\hat{\bm v}^i_m\times\hat{\bm v}^i_n)\big)\cdot \hat{\bm v}^i_l.\]
The remaining DOFs in $\tilde \Sigma_K$ are either identical with or linear combinations of those in $\tilde \Sigma_{\hat K}$. 
According to Proposition 3.4.7 in \cite{brenner-fe}, we complete the proof.
\end{proof}

To get the error estimate of the interpolation operator $\Lambda_K$, we introduce an interpolation  $I_K\bm w\in \bm P_{k-1}(k\geq 7)$ for $\bm w$ s.t.
\begin{itemize}
	\item Vertex DOFs: $i=1,2,3,4$ 
	\begin{itemize}
		\item $ (I_K\bm w)(  {\bm p}_i)=\bm w(  {\bm p}_i);$ 
		\item $D(I_K\bm w)(  {\bm p}_i)=D\bm w(  {\bm p}_i)$;
		\item $D^2(I_K\bm w)(  {\bm p}_i)=D^2\bm w(  {\bm p}_i)$.
	\end{itemize}	
	\item Edge DOFs:  $i=1,2,\cdots,6$
	\begin{itemize}
		\item $(I_K\bm w)(  {\bm e}^0_{ij})=\bm w({\bm e}^0_{ij}),\ j=1,2,\cdots,k-6;$
		\item  $\big(\nabla( I_K\bm w\cdot (\bm v^i_m\times\bm v^i_n))\cdot {\bm v}^i_l \big)(  {\bm e}^1_{ij})=\big(\nabla(\bm w\cdot (\bm v^i_m\times\bm v^i_n))\cdot {\bm v}^i_l \big)(  {\bm e}^1_{ij}),\ j=1,2,\cdots,k-5$, $\{l,[m,n]\}=\{2,[1,2]\}, \ \{2,[1,3]\}, \ \{2,[2,3]\}, \ \{3,[1,2]\},\ \{3,[1,3]\},\text{ or }\{3,[2,3]\}$,	\end{itemize}	
	where ${\bm v}^i_n\ (n=1,2,3)$, $ {\bm e}^0_{ij} \ ( j=1,2,\cdots,k-6)$, and $  {\bm e}^1_{ij}\ ( j=1,2,\cdots,k-5)$ are defined in Definition \ref{tetra-DOF}.
	\item Face DOFs:  $ i=1,2,3,4$
	\begin{itemize}
		\item $\frac{1}{area(f_i)}\int_{  f_i} I_K\bm w \cdot \bm{ q}\mathrm{d} S=\frac{1}{area(f_i)}\int_{  f_i} \bm w \cdot \bm{ q}\mathrm{d} S, \ \forall   {\bm q}\in P_{k-7}(f_i)\bm{t}_{i,1}\oplus P_{k-7}(f_i)\bm{t}_{i,2}$;
		\item $\int_{  f_i} I_K\bm w \cdot  {\bm n}_i  { q}\mathrm{d} S=\int_{  f_i} \bm w \cdot  {\bm n}_i  { q}\mathrm{d} S, \ \forall  {q}\in P_{k-7}(  f_i)$;
		\item $\det(B_K)\int_{f_i}\nabla\cdot I_K\bm wq\mathrm{d} S=\det(B_K)\int_{f_i}\nabla\cdot \bm wq\mathrm{d} S, \ \forall  {q}\in P_{k-5}(  f_i)$,
	\end{itemize}
	where ${\bm t}_{i,j}(j=1,2)$ is defined in Definition \ref{tetra-DOF}.	\item Interior DOFs: 
	\begin{itemize}
		\item $\int_{  K}I_K\bm w\cdot  {\bm q}\mathrm{d}  V=\int_{  K}\bm w\cdot  {\bm q}\mathrm{d}  V, \ \forall {\bm q}= \det(B_K)B_K^{-T}\hat{\bm q},\forall \hat{\bm q}\in \hat{\bm x}\times \bm P_{k-7}(\hat K)$;
		\item $\int_{  K}\nabla \cdot I_K\bm w { q}\mathrm{d}  V=\int_{  K}\nabla \cdot \bm w  { q}\mathrm{d}  V , \ \forall {q}= {\det(B_K)}^{-1}\hat { q },\ \forall \hat q \in P_{k-6}(\hat K)/\mathbb{R}$.
	\end{itemize}	
\end{itemize}
\begin{remark}
All the DOFs used to define $I_K$ are all the missing DOFs in Definition \ref{tetra-affine} because of $\nabla\cdot\nabla\times \bm u=0$ as well as those associated with $\nabla\times\bm u$.
\end{remark}

\begin{lemma}\label{err-affine-interp-curlu}
	If $\bm w\in \bm H^s(K)$ with $s>7/2$ and there exists a pair $\{m,q\}$ s.t.  $H^s(K) \hookrightarrow W^{m,q}(K)$, then we have the following error estimates for the interpolation $I_K$,
	\begin{align*}
	&\left\|\bm w-I_K\bm w\right\|_{m,q,K}\leq C|K|^{1/q-1/2}h_K^{s-m}\left\|\bm w\right\|_{s,K}.
	\end{align*}
\end{lemma}
\begin{proof}
The proof is standard, c.f. Theorem 3.1.4 in \cite{ciarlet2002finite}.
\end{proof}

\begin{lemma} \label{IK-LambdaK}	If $\nabla\times\bm u\in \bm H^s(K)$ with $s>7/2$ and there exists a pair $\{m,q\}$ s.t.  $H^s(K) \hookrightarrow W^{m,q}(K)$, then	\[\left\|I_K \nabla\times \bm u-\nabla\times\Lambda_K\bm u\right\|_{m,q,K}\leq C|K|^{1/q-1/2}h_K^{s-m}\left\|\nabla \times \bm  u\right\|_{s,K}.\]
\end{lemma}
\begin{proof}
For simplicity of notations, we let $\bm w=I_K \nabla\times \bm u-\nabla\times\Lambda_K\bm u$. Since $\bm w\in \bm P_{k-1}$, we have $\bm w=I_K\bm w=\sum c_i(\bm w)\bm N_i$, where $c_i(\bm w)$ are the DOFs to define $I_K\bm w$  and $\bm N_i$ are the corresponding basis functions. We first show all the DOFs  vanish except $\nabla(\bm w \cdot (\bm v^i_2\times\bm v^i_3))\cdot {\bm v}^i_3 $. Some of those are obvious 0 by the definition of $I_K$ and $\Lambda_K$.
We only  check the others. At first,
\begin{align*}
\bm w_{3z}&=(I_K \nabla\times \bm u-\nabla\times\Lambda_K\bm u)_{3z}
=(I_K \nabla\times \bm u)_{3z}+(\nabla\times\Lambda_K\bm u)_{1x}+(\nabla\times\Lambda_K\bm u)_{2y}\\&=( \nabla\times \bm u)_{3z}+(\nabla\times\bm u)_{1x}+(\nabla\times\bm u)_{2y}=0.
\end{align*}
\noindent
Similarly,
$\bm w_{1xx}=\bm w_{2yy}=\bm w_{3zz}=0.$\\
Applying integration by parts as well as  the definition of $I_K$ and $\Lambda_K$, we have
\begin{align*}
&\int_{  f_i}\bm w  \cdot  {\bm n}_i C\mathrm{d} S=\int_{  f_i} \nabla\times \bm u\cdot  {\bm n}_i C\mathrm{d} S-\int_{  f_i}(\nabla\times\Lambda_K\bm u)  \cdot  {\bm n}_i C\mathrm{d} S\\
=\int_{  f_i} \nabla\times& \bm u\cdot  {\bm n}_i C\mathrm{d} S-\int_{\partial f_i}\Lambda_K\bm u \cdot  {\bm \tau} C\mathrm{d} s=\int_{  f_i} \nabla\times \bm u\cdot  {\bm n}_i C\mathrm{d} S-\int_{\partial f_i}\bm u \cdot  {\bm \tau} C\mathrm{d} s=0.
\end{align*}
Finally, using the definition of $I_K$ and the fact $\nabla\cdot \nabla \times =0,$ we have
\begin{align*}
&\int_{f_i}\nabla\cdot\bm wq\mathrm{d} S =\int_{f_i}\nabla\cdot( \nabla\times \bm u-\nabla\times\Lambda_K\bm u)q\mathrm{d} S =0,\\
&\int_{  K}\nabla \cdot \bm w { q}\mathrm{d}  V=\int_{  K}\nabla \cdot ( \nabla\times \bm u-\nabla\times\Lambda_K\bm u) { q}\mathrm{d}  V=0.
\end{align*}
Now we estimate the non-vanishing term. By the definition of $I_K$, we have
\begin{align*}
&\big[\nabla(\bm w \cdot (\bm v^i_2\times\bm v^i_3))\cdot {\bm v}^i_3\big](\bm e_{ij}^1) =\big[\nabla\big(\nabla\times (\bm u-\Lambda_K\bm u) \cdot (\bm v^i_2\times\bm v^i_3)\big)\cdot {\bm v}^i_3\big](\bm e_{ij}^1).
\end{align*}
Since the divergence of $\nabla\times (\bm u-\Lambda_K\bm u)$ is 0, we can find 8 constants $C_{[m,n]}^l \ (1\leq l\leq 3, \ 1\leq m<n\leq 3\text{ except the case } l=3, m=2, \text{ and } n=3)$ independent of $h_K$ s.t.
\begin{align*}
&\quad\quad\quad\big[\nabla\big( \nabla\times (\bm u-\Lambda_K\bm u) \cdot (\bm v^i_2\times\bm v^i_3)\big)\cdot {\bm v}^i_3\big](\bm e_{ij}^1)\\
=&\sum _{\{l,[m,n]\}\neq \{3,[2,3]\}}C_{[m,n]}^l\big[\nabla\big( \nabla\times (\bm u-\Lambda_K\bm u) \cdot (\bm v^i_m\times\bm v^i_n)\big)\cdot {\bm v}^i_l\big](\bm e_{ij}^1),
\end{align*}
which can be finished by mapping to the reference element, finding the constants and then mapping back.
Furthermore, by the definition of $\Lambda_K\bm u$, we have
\begin{align*}
&\sum _{\{l,[m,n]\}\neq \{3,[2,3]\}}C_{[m,n]}^l\big[\nabla\big( \nabla\times (\bm u-\Lambda_K\bm u) \cdot (\bm v^i_m\times\bm v^i_n)\big)\cdot {\bm v}^i_l\big](\bm e_{ij}^1)\\
=&\sum _{1\leq m< n\leq 3}C_{[m,n]}^1\big[\nabla\big( \nabla\times (\bm u-\Lambda_K\bm u) \cdot (\bm v^i_m\times\bm v^i_n)\big)\cdot {\bm v}^i_1\big](\bm e_{ij}^1).
\end{align*}
Since $ \nabla\times \Lambda_K\bm u$ restricted on the edge $\bm  e_i$ is a polynomial vector of order $k-1$ which can be determined by 
all the vertex DOFs $\nabla\times \Lambda_K\bm u(\bm p_i)=I_K\nabla\times\bm u(\bm p_i)$, $D(\nabla\times \Lambda_K\bm u(\bm p_i))=D(I_K\nabla\times\bm u(\bm p_i))$, and $D^2(\nabla\times \Lambda_K\bm u(\bm p_i))=D^2(I_K\nabla\times\bm u(\bm p_i))$ and the edge DOFs $ \nabla\times \Lambda_K\bm u({\bm e}^0_{ij})=I_K\nabla\times\bm u({\bm e}^0_{ij})$, we arrive at
\[\sum _{1\leq m< n\leq 3}C_{[m,n]}^1\big[\nabla\big(( \nabla\times \bm u- I_K\nabla\times \bm u) \cdot (\bm v^i_m\times\bm v^i_n)\big)\cdot {\bm v}^i_1\big](\bm e_{ij}^1).\]
Therefore, by Lemma \ref{err-affine-interp-curlu}, we have
\begin{align}
&\big[\nabla\big(\nabla\times (\bm u-\Lambda_K\bm u) \cdot (\bm v^i_2\times\bm v^i_3)\big)\cdot {\bm v}^i_3\big](\bm e_{ij}^1)\nonumber\\
=&\sum _{1\leq m< n\leq 3}C_{[m,n]}^1\big[\nabla\big(( \nabla\times \bm u- I_K\nabla\times \bm u) \cdot (\bm v^i_m\times\bm v^i_n)\big)\cdot {\bm v}^i_1\big](\bm e_{ij}^1) \nonumber\\
\leq & C h_K^3\left|{\nabla}\times \bm u-I_K( {\nabla}\times \bm u)\right|_{1,\infty,K}\nonumber\\
\leq &C|K|^{-1/2}h_K^3{h_K^{s-1}}\left|{\nabla}\times\bm u\right|_{s,K}.\label{nonzero-estimate}
\end{align}
Suppose $\bm N_i$ are the basis functions associated with the non-vanishing DOFs. Then
\begin{align}\label{basis-estimate}
\left\|\bm  N_i\right\|_{m,q,K}\leq  Ch^{-2-m}|K|^{1/q}\left\|\hat{\bm  N_i}\right\|_{m,q,\hat K},
\end{align}
where $\hat{\bm N_i}=|B_K|B_K^{-1}\bm N_i$ are the basis functions on the reference element.

Combining \eqref{nonzero-estimate} and \eqref{basis-estimate}, we complete the proof.
\end{proof}

\begin{theorem}\label{err-affine-interp}
	If $\bm u\in \bm H^s(K)$ and $\nabla\times\bm u\in  \bm H^s(K)\cap \bm C^2(\bar K)$ with $s>\frac{5}{2}$, then we have the following error estimates for the interpolation $\Lambda_K$,
	\begin{align}
	&\left\|\bm u-\Lambda_K\bm u\right\|_K\leq Ch_K^{\min\{s,k\}}(\left\|\bm u\right\|_{s,K}+\left\|\nabla\times\bm u\right\|_{s,K}),\label{uerror}\\
	&\left\|\nabla\times(\bm u-\Lambda_K\bm u)\right\|_{m,q,K}\leq C|K|^{1/q-1/2}h_K^{s-m}\left\|\nabla \times \bm  u\right\|_{s,K}.
	\end{align}
\end{theorem}
\begin{proof}
	Due to the relationship $\Lambda_K \bm u\circ F_K=B_K^{-T}\Lambda_{\hat K}\hat {\bm u}$ obtained in Lemma \ref{relation}, the proof of \eqref{uerror} is standard, cf, Theorem 3.11 in \cite{curlcurl-conforming2D}. Combined Lemma \ref{err-affine-interp-curlu} and Lemma \ref{IK-LambdaK}, we obtain
	\begin{align*}
	\left\|\nabla\times(\bm u-\Lambda_K\bm u)\right\|_{m,q,K}&\leq \left\|\nabla\times\bm u-I_K\nabla\times\bm u)\right\|_{m,q,K}+\left\|I_K\nabla\times\bm u-\nabla\times\Lambda_K\bm u)\right\|_{m,q,K}\\
	&\leq C|K|^{1/q-1/2}h_K^{s-m}\left\|\nabla \times \bm  u\right\|_{s,K}.
	\end{align*}
	\end{proof}
\begin{theorem}\label{err-interp}
	If $\bm u\in \bm H^s(K)$ and $\nabla\times\bm u\in  \bm H^s(K)$ with $s>\frac{7}{2}$, then we have the following error estimates for the interpolation $\Pi_K$,
	\begin{align}
	&\left\|\bm u-\Pi_K\bm u\right\|_K\leq C{h_K^{{\min\{s,k\}}}}(\left\|\bm u\right\|_{s,K}+\left\|\nabla\times\bm u\right\|_{s,K}),\\
	&\left\|\nabla\times(\bm u-\Pi_K\bm u)\right\|_K\leq Ch_K^{\min\{s,k\}}\left\|\nabla\times\bm u\right\|_{s,K},\\
	&	\left\|(\nabla\times)^2(\bm u-\Pi_K\bm u)\right\|_K\leq Ch^{{\min\{s,k\}}-1}_K\left\|\nabla\times\bm u\right\|_{s,K}.
	\end{align}
\end{theorem}

\begin{proof}
   Since $\bm u-\Pi_K\bm u= \bm u-\Lambda_K \bm u+ \Lambda_K \bm u-\Pi_K\bm u$, it remains to estimate $\Lambda_K \bm u-\Pi_K\bm u$ in three different norms or semi-norms.
We denote $\bm  \Delta=\Lambda_K \bm u-\Pi_K\bm u$ which is a polynomial with a degree of no more than $7$. Also, the DOFs in $\tilde\Sigma_K$ for $\bm \Delta$ vanish except $\big(\nabla( {\nabla}\times {\bm  \Delta}\cdot (\bm v^i_m\times\bm v^i_n))\cdot {\bm v}^i_l \big)(  {\bm e}^1_{ij})$. 
Then 
\begin{align*}
\bm \Delta=\sum_{i=1}^6\sum_{j=1}^{k-5}\sum_{l=2}^3\sum_{[m,n]}\big(\nabla( {\nabla}\times {\bm  \Delta}\cdot (\bm v^i_m\times\bm v^i_n))\cdot {\bm v}^i_l \big)(  {\bm e}^1_{ij})\bm N_{ijl}^{[m,n]},
\end{align*}
where $\bm N_{ijl}^{[m,n]}$ are the basis functions of the finite elements defined in Definition \ref{tetra-affine} which are associated with the DOFs
$\big(\nabla( {\nabla}\times {\bm  \Delta}\cdot (\bm v^i_m\times\bm v^i_n))\cdot {\bm v}^i_l \big)(  {\bm e}^1_{ij})$.
Since ${\nabla}\times {\bm  \Delta}$ is divergence-free polynomial on the edge $\bm e_i$, $\nabla( {\nabla}\times {\bm  \Delta}\cdot \bm m)\cdot \bm m$, restricted on the edge $\bm e_i$, can be determined by the edge DOFs of the form $\nabla( {\nabla}\times {\bm  \Delta}\cdot \bm w_i)\cdot \bm w_j$. 
Hence, writing $\bm v^i_m\times\bm v^i_n$ and ${\bm v}^i_l$ as a linear combination of $\bm \tau,\bm m,\bm n$, we get
\begin{align*}
&\nabla( {\nabla}\times {\bm  \Delta}\cdot (\bm v^i_m\times\bm v^i_n))\cdot {\bm v}^i_l\\
=&\nabla\big( {\nabla}\times {\bm  \Delta}\cdot [(\bm v^i_m\times\bm v^i_n\cdot\bm \tau)\bm\tau+(\bm v^i_m\times\bm v^i_n\cdot\bm n)\bm n+(\bm v^i_m\times\bm v^i_n\cdot\bm m)\bm m]\big)\\
& \cdot \big(({\bm v}^i_l \cdot \bm \tau)\bm \tau+({\bm v}^i_l \cdot \bm n)\bm n+({\bm v}^i_l \cdot \bm m)\bm m\big)\\
=&\nabla\big( {\nabla}\times {\bm  \Delta}\cdot [(\bm v^i_m\times\bm v^i_n\cdot\bm \tau)\bm\tau+(\bm v^i_m\times\bm v^i_n\cdot\bm n)\bm n+(\bm v^i_m\times\bm v^i_n\cdot\bm m)\bm m]\big)\\
& \cdot \big(({\bm v}^i_l \cdot \bm n)\bm n+({\bm v}^i_l \cdot \bm m)\bm m\big)\\
=&\nabla\big( {\nabla}\times( \bm u-\Lambda_K\bm u)\cdot [(\bm v^i_m\times\bm v^i_n\cdot\bm \tau)\bm\tau+(\bm v^i_m\times\bm v^i_n\cdot\bm n)\bm n+(\bm v^i_m\times\bm v^i_n\cdot\bm m)\bm m]\big)\\
& \cdot \big(({\bm v}^i_l \cdot \bm n)\bm n+({\bm v}^i_l \cdot \bm m)\bm m\big).
\end{align*}
Each term has the following estimate. We only show the first term
	\begin{align*}
	&\nabla\big( {\nabla}\times( \bm u-\Lambda_K\bm u)\cdot ((\bm v^i_m\times\bm v^i_n)\cdot\bm \tau)\bm\tau\big)\cdot ({\bm v}^i_l \cdot \bm n)\bm n\\
	\leq &
	Ch_K^3\left|{\nabla}\times( \bm u-\Lambda_K\bm u)\right|_{1,\infty,K}\\
	\leq& C|K|^{-1/2}h_K^3{h_K^{s-1}}\left|{\nabla}\times\bm u\right|_{s,K}.
	\end{align*}
According to the mapping \eqref{mapping-u}, the basis functions $\bm N_{ijl}^{[m,n]}$ satisfy
\begin{align*}
&\left\|\bm N^{[m,n]}_{ijl}\right\|\leq C{h_K^{1/2}}\left\|\hat {\bm N}^{[m,n]}_{ijl}\right\|,\\
&\left\|\nabla\times \bm N^{[m,n]}_{ijl}\right\|\leq C{h_K^{-1/2}}\left\|\hat{ \nabla}\times\hat {\bm N}^{[m,n]}_{ijl}\right\|,\\
&\left\|\nabla\times\nabla\times \bm N^{[m,n]}_{ijl}\right\|\leq C{h_K^{-3/2}}\left\|\hat{ \nabla}\times\hat{ \nabla}\times\hat {\bm N}^{[m,n]}_{ijl}\right\|,
\end{align*}
where $\hat {\bm N}^{[m,n]}_{ijl}$ are the corresponding basis functions on $\hat K$ and satisfy ${\bm N}^{[m,n]}_{ijl}=B_K^{-T}\hat {\bm N}^{[m,n]}_{ijl}$.
By combining the above estimates, we complete the proof.
\end{proof}

\section{Numerical Experiments}
In this section, we use the $H(\text{curl}^2)$-conforming finite elements developed in Section 3 to solve the quad-curl problem which is introduced as: for $\bm  f\in H(\td^0;\Omega)$, find $\bm u$ s.t.
\begin{equation}\label{prob1}
\begin{split}
(\nabla\times)^4\bm u+\bm u&=\bm f\ \ \text{in}\;\Omega,\\
\nabla \cdot \bm u &= 0\ \ \text{in}\;\Omega,\\
\bm u\times\bm n&=0\ \ \text{on}\;\partial \Omega,\\
\nabla \times \bm u&=0\ \  \text{on}\;\partial \Omega,
\end{split}
\end{equation}
where $\Omega \in\mathbb{R}^3$ is a contractable Lipschitz domain and  $\bm n$ is the unit outward normal vector to $\partial \Omega$.
Divergence-free condition $\nabla\cdot\bm u=0$ satisfies automatically, since we have the lower-order term $\bm u$ in the equation \eqref{prob1}.
The variational formulation is to
find $\bm u\in H_0(\tc^2;\Omega)$,  s.t.
\begin{equation}\label{prob22}
\begin{split}
a(\bm u,\bm v)&=(\bm f, \bm v)\quad \forall \bm v\in H_0(\tc^2;\Omega),
\end{split}
\end{equation}
with $a(\bm u,\bm v)=(\nabla\times\nabla\times\bm u,\nabla\times\nabla\times\bm v) + (\bm u,\bm v)$.

Let \,$\mathcal{T}_h\,$ be a partition of the domain $\Omega$
consisting of shape-regular tetrahedra. We denote by $h_K$ the diameter of each element $K \in
\mathcal{T}_h$ and by $h$ the mesh size of $\mathcal {T}_h$. We define
\begin{eqnarray*}
	&&  V_h=\{\bm{v}_h\in H(\text{curl}^2;\Omega):\ \bm v_h|_K\in \mathcal{R}_k,\ \forall K\in\mathcal{T}_h\}.\\
	&&   V^0_h=\{\bm{v}_h\in V_h,\ \bm{n} \times \bm{v}_h=0\ \text{and}\ \nabla\times  \bm{v}_h = 0 \ \text {on} \ \partial\Omega\}.
\end{eqnarray*}

Now we define a global interpolation operator $\Pi_h$ which is defined piecewisely:
\[\Pi_h|_K=\Pi_K.\]
The global interpolation has the following error estimates.
\begin{theorem}\label{global-err-interp}
	If $\bm u\in \bm H^s(\Omega)$ and $\nabla\times\bm u\in  \bm H^s(\Omega)$ with $s>\frac{7}{2}$, then we have the following error estimates for the interpolation $\Pi_h$,
	\begin{align}
	&\left\|\bm u-\Pi_h\bm u\right\|\leq C{h^{\min\{s,k\}}}(\left\|\bm u\right\|_{s}+\left\|\nabla\times\bm u\right\|_{s}),\\
	&\left\|\nabla\times(\bm u-\Pi_h\bm u)\right\|\leq Ch^{\min\{s,k\}}\left\|\nabla\times\bm u\right\|_{s},\\
	&	\left\|(\nabla\times)^2(\bm u-\Pi_h\bm u)\right\|\leq Ch^{{\min\{s,k\}}-1}\left\|\nabla\times\bm u\right\|_{s}.
	\end{align}
\end{theorem}
The Theorem is proved by the fact $\left\|(\nabla\times)^i(\bm u-\Pi_h\bm u)\right\|=\sum_{K\in \mathcal T_h}\left\|(\nabla\times)^i(\bm u-\Pi_K\bm u)\right\|_K$ and Theorem \ref{err-interp}.

The $H(\text{curl}^2)$-conforming finite element method seeks $\bm u_h\in V^0_h$,  s.t.
\begin{equation}\label{prob3}
\begin{split}
 a(\bm u_h,\bm v_h)&=(\bm f, \bm v_h)\quad \forall \bm v_h\in V^0_h.
\end{split}
\end{equation}

To implement the boundary conditions, we can either let all the DOFs which yield   boundary conditions be 0 or introduce two Lagrange multipliers.

\subsection{Example 1}

We consider the problem \eqref{prob1} on a unit cube $\Omega=(0,1)\times(0,1)\times(0,1)$ with the exact solution
	\begin{equation}
	\bm u=\left(
	\begin{array}{c}
	0\\
	3\pi\sin^3(\pi x)\sin^3(\pi y)\sin^2(\pi z)\cos(\pi z) \\
	-3\pi \sin^3(\pi x)\sin^3(\pi z)\sin^2(\pi y)\cos(\pi y)\\
	\end{array}
	\right).
	\end{equation}
   Then  the source term $\bm f$ can be obtained by a simple calculation. 
	Denote
	\[\bm e_h=\bm u-\bm u_h.\]
	We partition the unit cube into $N^3$ small cubes and then partition each small cube into 6 congruent tetrahedra. Varying $h=1/N$ from ${1}/{2}$ to ${1}/{8}$, Table \ref{tab1} illustrates the errors and convergence rates of $\bm u_h$ with $k=7$ in several different norms. We can observe the convergence rates of 7 in $H(\text{curl})$ norm and of
	6 in 	$H(\text{curl}^2)$ norm which coincide with the convergence orders of the interpolation $\Pi_h$.
		
\begin{table}[h]
	\centering
	\caption{Example 1: Numerical results by the lowest-order tetrahedral $H(\text{curl}^2)$ element} \label{tab1}
	\begin{adjustwidth}{0cm}{0cm}
			\resizebox{!}{2.6cm}
			{
	\begin{tabular}{ccccccc}
		\hline
		$h$ &$\left\|\bm u-\bm u_h\right\|$& rates&$\left\|\nabla\times\bm u-\nabla\times\bm u_h\right\|$& rates&$\left\|(\nabla\times)^2\bm u-(\nabla\times)^2\bm u_h\right\|$& rates\\
		\hline
		$1\slash2$&3.8334785395e+00 &   10.8753       & 8.0089356298e-01&5.1543
		&1.6715185815e+01&4.0572\\
		$1\slash3$&4.6617638169e-02  & 6.6651 &9.9072060818e-02   &5.1588 
		&3.2261165763e+00 & 4.4177 \\
		$1\slash4$&6.8520104719e-03&  9.6538   &2.2460507680e-02 &  6.7446 &9.0519796164e-01& 5.5650\\
		$1\slash5$&7.9482178822e-04& 9.7577&4.9865729850e-03  &7.4171
		&2.6148268239e-01 & 6.0130\\
		$1\slash6$&1.3416712567e-04 &  9.7979  &1.2897568262e-03    & 7.3648
		&8.7363073645e-02& 5.9716\\
		$1\slash7$&2.9628521344e-05& 9.4904&4.1443815436e-04 & 7.1025
		&3.4797494022e-02& 5.8304\\
		$1\slash8$&8.3433920597e-06  & & 1.6053634598e-04   &
		&1.5974611799e-02 & \\
		\hline
	\end{tabular}
	}
	\end{adjustwidth}
\end{table}
\subsection{Example 2}
In this example, we consider the problem \eqref{prob1} on a 
unit cube with the source term $\bm f=(1,1,1)^T.$
In this case, we can not express the exact solution explicitly, so we  seek an approximation of $\3bar\bm u-\bm u_h\3bar=\sqrt{a(\bm u-\bm u_h,\bm u-\bm u_h)}$.
Due to the  orthogonality $a(\bm u-\bm u_h,\bm u_h)=0$,
we have
\[\3bar\bm u-\bm u_h\3bar^2=\3bar\bm u\3bar^2-2a(\bm u,\bm u_h)+\3bar\bm u_h\3bar^2=\3bar\bm u\3bar^2-\3bar\bm u_h\3bar^2.\]
Since $\bm u_h\in V_h\subset V_{h/2}$, $a(\bm u-\bm u_{h/2},\bm u_h)=0$ and then
\begin{align*}
	\3bar\bm u_{h/2}-\bm u_h\3bar^2=&\3bar\bm u_{h/2}\3bar^2-2a(\bm u_{h/2},\bm u_h)+\3bar\bm u_h\3bar^2\\
=&\3bar\bm u_{h/2}\3bar^2-2a(\bm u,\bm u_h)+\3bar\bm u_h\3bar^2\\
=&\3bar\bm u_{h/2}\3bar^2-\3bar\bm u_h\3bar^2.
\end{align*}
Thanks to $a(\bm u-\bm u_h,\bm u-\bm u_{h/2})=a(\bm u-\bm u_{h/2},\bm u-\bm u_{h/2})$ since $a(\bm u_h,\bm u-\bm u_{h/2})=a(\bm u_{h/2},\bm u-\bm u_{h/2})=0,$ we have
\begin{align*}
	\3bar\bm u_{h/2}-\bm u_h\3bar^2&=\3bar\bm u_{h/2}-\bm u\3bar^2+2a(\bm u_{h}-\bm u,\bm u-\bm u_{h/2})+\3bar\bm u-\bm u_h\3bar^2\\
	=&-\3bar\bm u_{h/2}-\bm u\3bar^2+\3bar\bm u-\bm u_h\3bar^2
\approx \3bar\bm u-\bm u_h\3bar^2.
\end{align*}
Now we can treat $\3bar\bm u_{h/2}-\bm u_h\3bar^2=\3bar\bm u_{h/2}\3bar^2-\3bar\bm u_h\3bar^2$ as an approximation of $\3bar\bm u-\bm u_h\3bar^2$. The Table \ref{tab2} shows that the convergence rates in energy norm are about 2. The convergence rates deteriorate due to the poor solution regularity.

We also draw the graph of the numerical solution at $y=0.1$ (which is close to the boundary). From the graph \ref{fig1}, we do not observe any oscillation phenomenon,  which indicates the boundary conditions are implemented correctly.
\begin{table}[h]
	\centering
	\caption{Example 2: Numerical results by the lowest-order  tetrahedral  $H(\text{curl}^2)$ element} \label{tab2}
	\begin{adjustwidth}{0cm}{0cm}
			\resizebox{!}{1.75cm}
			{
	\begin{tabular}{ccccccc}
		\hline
		$h$ &$\left\|\bm u_h\right\|$&$\left\|\nabla\times\bm u_h\right\|$&$\left\|(\nabla\times)^2\bm u_h\right\|$&$\3bar \bm u-\bm u_h\3bar$& rates\\
		\hline
		$1\slash1$ &4.0503711308e-04& 2.1012866605e-03 &2.2019421906e-02
		&1.9004476086e-02& \\
		$1\slash2$&6.8754227877e-04& 3.4074245801e-03&2.8957231505e-02& 1.9895511952e-03 
		&3.2558 \\
		$1\slash4$&6.8874370251e-04&3.4044210424e-03&2.9025822581e-02&  3.6939587847e-04&2.4292\\
		$1\slash8$&6.8880221694e-04& 3.4044347844e-03&2.9028170036e-02&& \\
		\hline
	\end{tabular}
	}
	\end{adjustwidth}
\end{table}

\begin{figure}
	\includegraphics[scale=0.2]{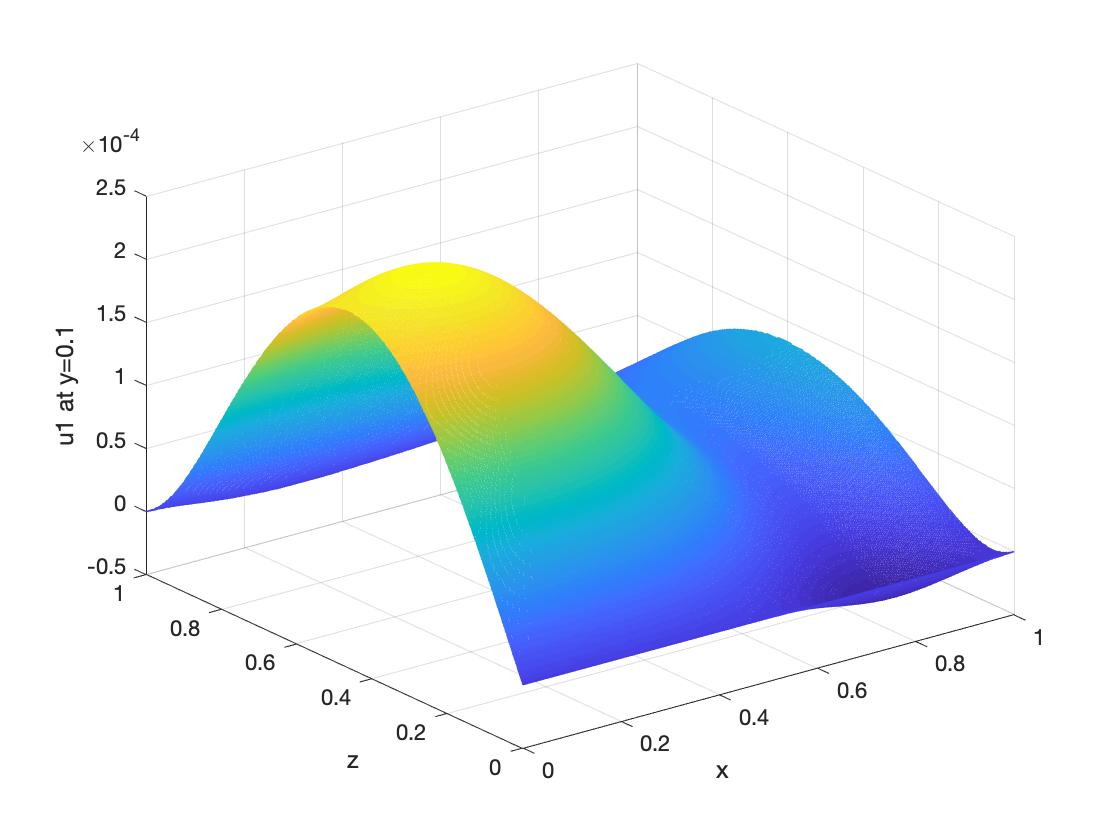}
	\includegraphics[scale=0.2]{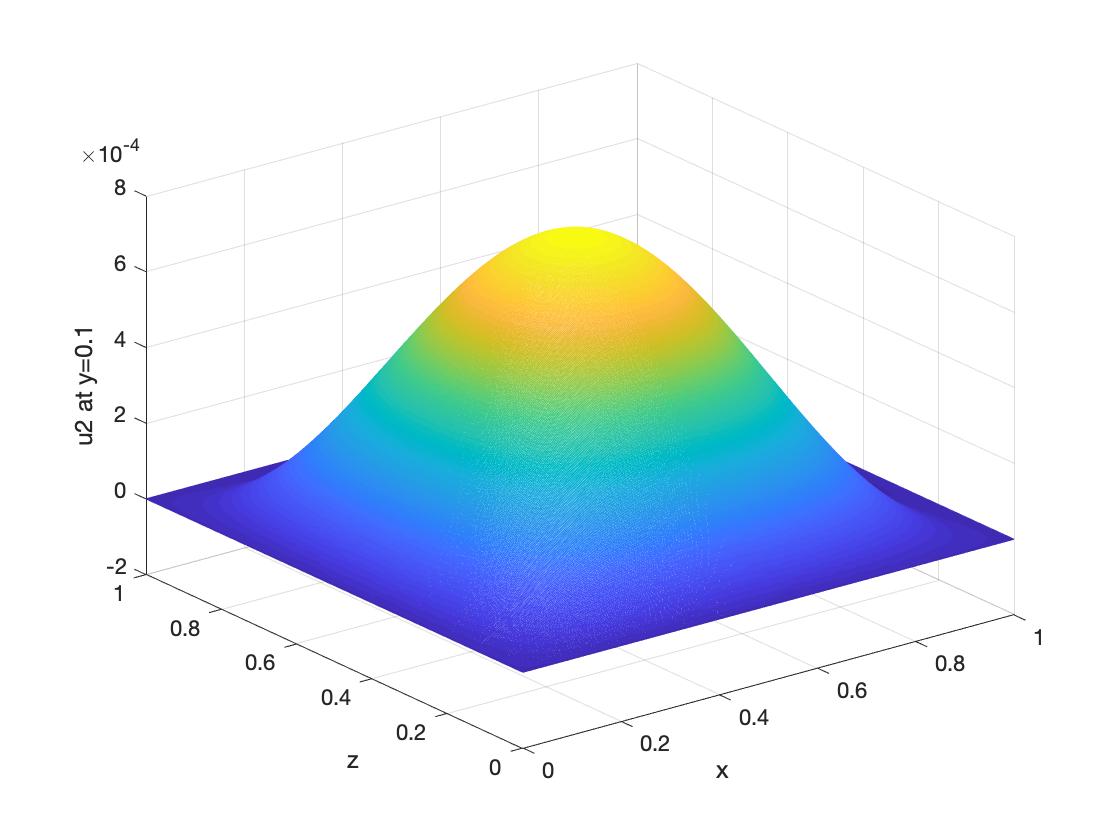}
	\includegraphics[scale=0.2]{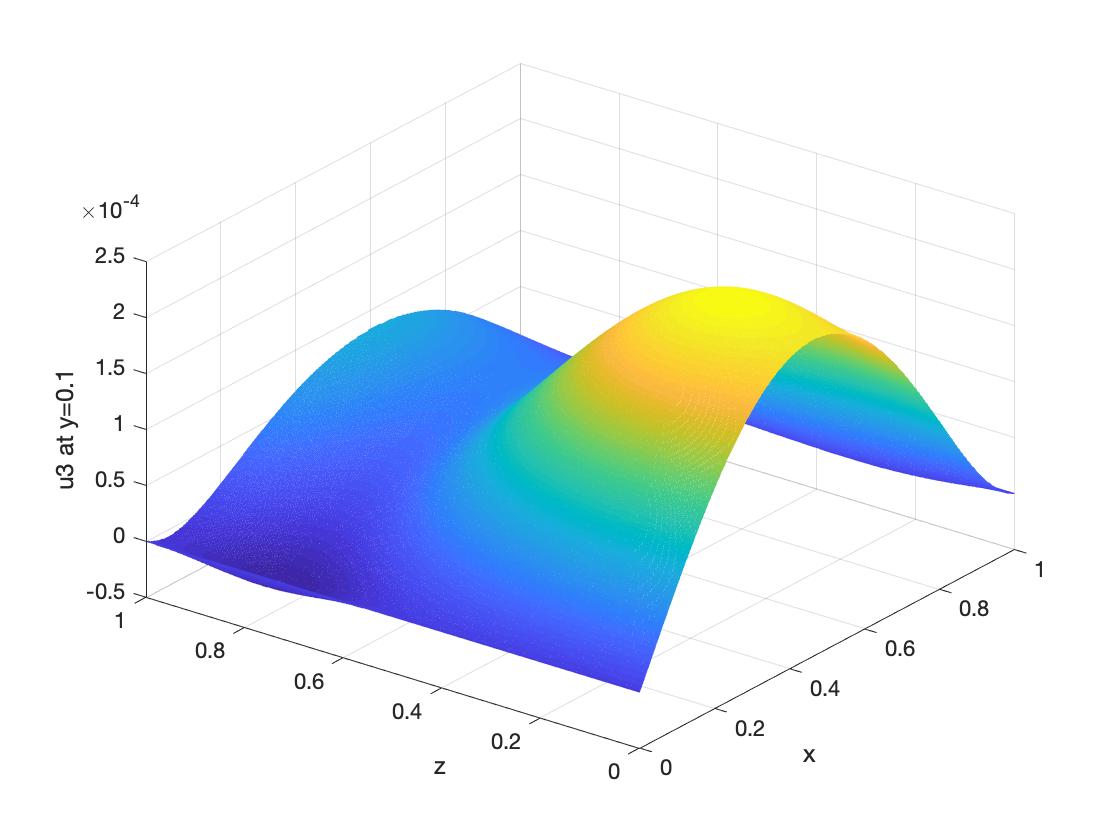}\\
	\caption{The numerical solution  $\bm u_h$ at $y=0.1$.}\label{fig1}
\end{figure}

\subsection{Example 3}
We also consider the problem \eqref{prob1} on an L-shape domain $\Omega=(0,1)\times(0,1)\times(0,1)\slash[0.5,1)\times(0,0.5]\times[0,1]$ with the source term $\bm f=(1,1,1)^T$.
	
Table \ref{tab3} illustrates errors and convergence rates of $\bm u_h$ in this case.  Due to the singularity of the domain, convergence rates deteriorate sharply.

We should note that it is not very precise to use the mesh sizes $1,1/2,1/4,1/5$ for estimating the errors in the sense of energy norm. However, we only can refine the mesh to 1/5 because of the huge number of DOFs. We just use these data to estimate the errors approximately.

\begin{table}[h]
	\centering
	\caption{Example 3: Numerical results by the lowest-order  tetrahedral  $H(\text{curl}^2)$ element} \label{tab3}
	\begin{adjustwidth}{0cm}{0cm}
			\resizebox{!}{1.75cm}
			{
	\begin{tabular}{ccccccc}
		\hline
		$h$ &$\left\|\bm u_h\right\|$&$\left\|\nabla\times\bm u_h\right\|$&$\left\|(\nabla\times)^2\bm u_h\right\|$&$\3bar \bm u-\bm u_h\3bar$& rates\\
		\hline
		
		$1\slash2$& 1.4343168282e-04& 8.3323524761e-04 &1.2019789956e-02&  2.6363756085e-03
		&1.0795 \\
		$1\slash3$&8.7241298710e-05&5.0404132111e-04&9.4396047316e-03 &  1.7018568232e-03&1.2624\\
		$1\slash4$&5.8426817327e-05& 3.3808222459e-04&1.2795582951e-03 & 1.1836017614e-03
& \\
		$1\slash5$&4.1850682102e-05 & 2.4212569094e-04 &6.5995367635e-03&& \\

		\hline
	\end{tabular}
	}
	\end{adjustwidth}
\end{table}

\section{Conclusion}
In this paper, we construct and analyze, for the first time, the tetrahedral $H(\tc^2)$-conforming elements. We employ our new elements to solve the quad-curl problem  in  numerical experiments. It turns out that our new elements work well for solving quad-curl problem. However, the elements have a great number of DOFs, which makes it difficult to get a family of very accurate basis functions and expensive to refine the grid further. In our future research, we will construct hierarchical basis functions of curl-curl conforming elements and try to decrease the number of DOFs.

%
%
%
%
\section*{Appendix:An Efficient Implementation of the Tetrahedral Elements}
Since the finite elements proposed in this article involve normal derivatives to edges, we can not relate the basis functions on a general element and those on the reference element by the mapping \eqref{mapping-u}. In addition, it is difficult to obtain the basis functions on each general element by solving a large-scale (for the lowest-order case, it is 315$\times$315) and ill-conditioned matrix. Hence we apply the method proposed in \cite{Dom2008Algorithm} for the Argyris element to construct our basis functions on a general element. In this section, we will demonstrate this method for the lowest-order case.

Suppose $\bm x_{\alpha}, \alpha=1,2,3,4$, $\bm e_{\alpha}, \alpha=1,2,\cdots,6$, and $f_\alpha,\alpha=1,2,3,4$ are the 4 vertex, 6 edges, and 4 faces of a general element $K$. Suppose also $\bm m_\alpha$ is the midpoint of the edge $\bm e_\alpha$ and  $\bm e_{\alpha,j},j=1,2$ are the two tripartite points of the edge $\bm e_\alpha$. 

We define
\begin{align*}
&L^{v}_{\alpha,i}(\bm \phi)=(\nabla\times\bm \phi)_1(\bm x_{\alpha}),\ i=1,2,3,\ \alpha=1,2,3,4.\\	
&L^{o,v}_{\alpha,i}(\bm \phi)=\partial_o(\nabla\times\bm \phi)_1(\bm x_{\alpha}),\ i=1,2,3,\ \alpha=1,2,3,4, \ o\in \{x,y,z\} \text{ if } i=1,2,\text{ otherwise}, \\
&o\in \{x,y\}.\\
&L^{o,v}_{\alpha,i}(\bm \phi)=\partial_o^2(\nabla\times\bm \phi)_1(\bm x_{\alpha}), \ i=1,2,3, \ \alpha=1,2,3,4,\
    o\in 
    \begin{cases}
    	\{yy,zz,xz,xy,yz\}, \ i=1,\\
    	\{xx,zz,xz,xy,yz\},\ i=2,\\
    	\{xx,yy,xz,xy,yz\},\ i=3.
    \end{cases}\\
&L^{1,e}_{\alpha,i}(\bm \phi)=(\nabla\times\bm \phi)_1(\bm m_{\alpha}),\ i=1,2,3,\ \alpha=1,2,\cdots,6.\\	
&L^{\bm a_i,\bm n}_{\alpha,i,j}(\bm \phi)=(\nabla(\nabla\times\bm \phi\cdot\bm a_i)\cdot \bm n)(\bm e_{\alpha,j}),\ i=1,2,3,\ \alpha=1,2,\cdots,6, \ j=1,2,
\text{ and } \bm a_1=\bm \tau, \\
& \bm a_2=\bm n, \bm a_3=\bm m,\\
&\text{ where }  \bm n,\bm m \text{ and } \bm \tau \text{ are } 
\text{ two unit normal vectors and the unit tangent vector to } \bm e_{\alpha}. \\
 &L^{\bm a_i,\bm m}_{\alpha,i,j}(\bm \phi)=(\nabla(\nabla\times\bm \phi\cdot\bm a_i)\cdot \bm m)(\bm e_{\alpha,j}),\ i=1,2,\ \alpha=1,2,\cdots,6, \ j=1,2,\\
 &\text{where } \bm a_1=\bm \tau \text{ and } \bm a_2=\bm n.\\
&L^{1,f}_{\alpha,i}(\bm \phi)=\frac{1}{area(f_i)}\int_{f_i}\nabla\times\bm \phi\cdot\bm \tau_i	\mathrm{d}S,\ i=1,2,\ \alpha=1,2,3,4, \text{ and } \bm \tau_{i} \text{ is the unit tangent} \\
&\text{  vector on the face }  f_{\alpha}. \\
&L^{0,e}_{\alpha,i}(\bm \phi)=\int_{\bm e_{\alpha}}\bm \phi\cdot\bm \tau_{\alpha}q_i\mathrm{d}s,\ \alpha=1,2,\cdots,6,\ i=1,2,\cdots,7.
\\
&L^{0,f}_{\alpha,i}(\bm \phi)=\frac{1}{area(f_i)}\int_{ f_{\alpha}}\bm \phi\cdot\bm q_i\mathrm{d}S, \ \alpha=1,2,3,4,\ i=1,2,\cdots,15.\\
&L^{0,K}_{i}(\bm \phi)=\int_{K}\bm \phi\cdot\bm q_i\mathrm{d}V, \ i=1,2,\cdots,20.\\
&L^{1,K}_{i}(\bm \phi)=\int_{K}\nabla\times\bm \phi\cdot\bm q_i\mathrm{d}V,\  i=1,2,3.
 \end{align*}
 
We list these functionals as $L_j$ for $j = 1,2, \cdots,315$ in the following order:
\begin{align*}
&L^v_{1,i},i=1,2,3,\ L^v_{2,i},i=1,2,3,\ L^v_{3,i},i=1,2,3,\ L^v_{4,i},i=1,2,3, 
	\\ 
&L^{o}_{1,1} (\text{ the elements in } o \text{ are ordered as }  x,y,z), \\
&L^{o}_{1,2}(\text{ the elements in } o \text{ are ordered as }  x,y,z),\\
&L^{o}_{1,3} (\text{ the elements in } o \text{ are ordered as }  x,y),\\
&L^{o}_{2,1},L^{o}_{2,2},L^{o}_{2,3},L^{o}_{3,1},L^{o}_{3,2},L^{o}_{3,3},L^{o}_{4,1},L^{o}_{4,2},L^{o}_{4,3},\\
&L^{o}_{1,1} (\text{ the elements in } o \text{ are ordered as }  yy,zz,xz,xy,yz), \\
&L^{o}_{1,2}(\text{ the elements in } o \text{ are ordered as }  xx,zz,xz,xy,yz),\\
&L^{o}_{1,3}(\text{ the elements in } o \text{ are ordered as }  xx,yy,xz,xy,yz),\\
&L^{o}_{2,1},L^{o}_{2,2},L^{o}_{2,3},
L^{o}_{3,1},L^{o}_{3,2},L^{o}_{3,3},
L^{o}_{4,1},L^{o}_{4,2},L^{o}_{4,3},\\
&L^{1,e}_{1,i}, i=1,2,3, \ L^{1,e}_{2,i}, i=1,2,3,\ L^{1,e}_{3,i}, i=1,2,3,  \\
& L^{1,e}_{4,i}, i=1,2,3, \ L^{1,e}_{5,i}, i=1,2,3, \ L^{1,e}_{6,i}, i=1,2,3.\\
&L^{\bm \tau,\bm n}_{1,1},L^{\bm n,\bm n}_{1,1},L^{\bm m,\bm n}_{1,1},L^{\bm \tau,\bm m}_{1,1},L^{\bm n,\bm m}_{1,1},L^{\bm \tau,\bm n}_{1,2},L^{\bm n,\bm n}_{1,2},L^{\bm m,\bm n}_{1,2},L^{\bm \tau,\bm m}_{1,2},L^{\bm n,\bm m}_{1,2},\\
&L^{\bm \tau,\bm n}_{2,1},L^{\bm n,\bm n}_{2,1},L^{\bm m,\bm n}_{2,1},L^{\bm \tau,\bm m}_{2,1},L^{\bm n,\bm m}_{2,1},L^{\bm \tau,\bm n}_{2,2},L^{\bm n,\bm n}_{2,2},L^{\bm m,\bm n}_{2,2},L^{\bm \tau,\bm m}_{2,2},L^{\bm n,\bm m}_{2,2},\\
&L^{\bm \tau,\bm n}_{3,1},L^{\bm n,\bm n}_{3,1},L^{\bm m,\bm n}_{3,1},L^{\bm \tau,\bm m}_{3,1},L^{\bm n,\bm m}_{3,1},L^{\bm \tau,\bm n}_{3,2},L^{\bm n,\bm n}_{3,2},L^{\bm m,\bm n}_{3,2},L^{\bm \tau,\bm m}_{3,2},L^{\bm n,\bm m}_{3,2},\\
&L^{\bm \tau,\bm n}_{4,1},L^{\bm n,\bm n}_{4,1},L^{\bm m,\bm n}_{4,1},L^{\bm \tau,\bm m}_{4,1},L^{\bm n,\bm m}_{4,1},L^{\bm \tau,\bm n}_{4,2},L^{\bm n,\bm n}_{4,2},L^{\bm m,\bm n}_{4,2},L^{\bm \tau,\bm m}_{4,2},L^{\bm n,\bm m}_{4,2},\\
&L^{\bm \tau,\bm n}_{5,1},L^{\bm n,\bm n}_{5,1},L^{\bm m,\bm n}_{5,1},L^{\bm \tau,\bm m}_{5,1},L^{\bm n,\bm m}_{5,1},L^{\bm \tau,\bm n}_{5,2},L^{\bm n,\bm n}_{5,2},L^{\bm m,\bm n}_{5,2},L^{\bm \tau,\bm m}_{5,2},L^{\bm n,\bm m}_{5,2},\\
&L^{\bm \tau,\bm n}_{6,1},L^{\bm n,\bm n}_{6,1},L^{\bm m,\bm n}_{6,1},L^{\bm \tau,\bm m}_{6,1},L^{\bm n,\bm m}_{6,1},L^{\bm \tau,\bm n}_{6,2},L^{\bm n,\bm n}_{6,2},L^{\bm m,\bm n}_{6,2},L^{\bm \tau,\bm m}_{6,2},L^{\bm n,\bm m}_{6,2},\\
&L_{1,i}^{1,f},i=1,2,\ L_{2,i}^{1,f},i=1,2,
\ L_{3,i}^{1,f},i=1,2, \ L_{4,i}^{1,f},i=1,2,\\
& L^{0,e}_{1,i},i=1,2,\cdots,7,\ L^{0,e}_{2,i},i=1,2,\cdots,7,\ L^{0,e}_{3,i},i=1,2,\cdots,7,\\
& L^{0,e}_{4,i},i=1,2,\cdots,7, \ L^{0,e}_{5,i},i=1,2,\cdots,7, \ L^{0,e}_{6,i},i=1,2,\cdots,7,\\
& L^{0,f}_{1,i},i=1,2,\cdots,15, \ L^{0,f}_{2,i},i=1,2,\cdots,15, \ L^{0,f}_{3,i},i=1,2,\cdots,15, \ L^{0,f}_{4,i},i=1,2,\cdots,15,\\
& L_i^{0,K},i=1,2,\cdots,20, \ L_i^{1,K},i=1,2,3.
\end{align*}
The functionals $\hat L_i, i=1,2,\cdots,315$ are the counterparts on the reference element.
The basis functions $\{N_j\}_{j=1}^{315}$ for the finite element on a general element $K$ satisfy 
\[ L_i(N_j)=\delta_{ij}\ \ i,j\in{1,2,\cdots,315}.\]
The basis functions $\{\hat N_j\}_{j=1}^{315}$ for the finite element on the reference element $\hat K$ satisfy 
\[ \hat L_i(\hat N_j)=\delta_{ij}\ \ i,j\in{1,2,\cdots,315}.\]
Note that $N_j$ and $\hat N_j$ can not be related with mapping \eqref{mapping-u}.
We define 
     \[ \tilde L_i(\bm \phi)=\hat L_i(B_K^T\bm \phi\circ F).\]
 Since both sets ${L_i}$ and ${\tilde L_i}$ are bases of $\mathcal{R}_7^*$, the dual space to $\mathcal{R}_7$, there exists a matrix $C = (c_{ij})$ such that
\[\tilde L_i=\sum_{j=1}^{315}c_{ij}L_j\text{ in } \mathcal R_7^*.\]
By an elementary transposition argument, it follows that
\[N_i\circ F=\sum_{k=1}^{315}c_{ki}\hat N_k\text{ in } \mathcal R_7.\]
If we have obtained the basis functions on the reference element and the matrix C, we then obtain the basis functions on a general element.
We introduce a new set of functionals $L^*_i,i = 1,2,\cdots,383,$ which are listed in order as follows.
\begin{align*}
&L^v_{1,i},i=1,2,3,\ L^v_{2,i},i=1,2,3,\ L^v_{3,i},i=1,2,3,\ L^v_{4,i},i=1,2,3, 
	\\ 
&L^{o}_{1,1} (\text{ the elements in } o \text{ are ordered as }  x,y,z), \\
&L^{o}_{1,2}(\text{ the elements in } o \text{ are ordered as }  x,y,z),\\
&L^{o}_{1,3} (\text{ the elements in } o \text{ are ordered as }  x,y,z),\\
&L^{o}_{2,1},L^{o}_{2,2},L^{o}_{2,3},L^{o}_{3,1},L^{o}_{3,2},L^{o}_{3,3},L^{o}_{4,1},L^{o}_{4,2},L^{o}_{4,3},\\
&L^{o}_{1,1} (\text{ the elements in } o \text{ are ordered as }  xx,yy,zz,xz,xy,yz), \\
&L^{o}_{1,2}(\text{ the elements in } o \text{ are ordered as }  xx,yy,zz,xz,xy,yz),\\
&L^{o}_{1,3}(\text{ the elements in } o \text{ are ordered as }  xx,yy,zz,xz,xy,yz),\\
&L^{o}_{2,1},L^{o}_{2,2},L^{o}_{2,3},
L^{o}_{3,1},L^{o}_{3,2},L^{o}_{3,3},
L^{o}_{4,1},L^{o}_{4,2},L^{o}_{4,3},\\
&L^{1,e}_{1,i}, i=1,2,3, \ L^{1,e}_{2,i}, i=1,2,3,\ L^{1,e}_{3,i}, i=1,2,3,  \\
& L^{1,e}_{4,i}, i=1,2,3, \ L^{1,e}_{5,i}, i=1,2,3, \ L^{1,e}_{6,i}, i=1,2,3,\\
&L^{\bm \tau,\bm \tau}_{1,1},L^{\bm n,\bm \tau}_{1,1},L^{\bm m,\bm \tau}_{1,1},L^{\bm \tau,\bm n}_{1,1},L^{\bm n,\bm n}_{1,1},L^{\bm m,\bm n}_{1,1},L^{\bm \tau,\bm m}_{1,1},L^{\bm n,\bm m}_{1,1},L^{\bm m,\bm m}_{1,1},\\
&L^{\bm \tau,\bm \tau}_{1,2},L^{\bm n,\bm \tau}_{1,2},L^{\bm m,\bm \tau}_{1,2},L^{\bm \tau,\bm n}_{1,2},L^{\bm n,\bm n}_{1,2},L^{\bm m,\bm n}_{1,2},L^{\bm \tau,\bm m}_{1,2},L^{\bm n,\bm m}_{1,2},L^{\bm m,\bm m}_{1,2},\\
&L^{\bm \tau,\bm \tau}_{2,1},L^{\bm n,\bm \tau}_{2,1},L^{\bm m,\bm \tau}_{2,1},L^{\bm \tau,\bm n}_{2,1},L^{\bm n,\bm n}_{2,1},L^{\bm m,\bm n}_{2,1},L^{\bm \tau,\bm m}_{2,1},L^{\bm n,\bm m}_{2,1},L^{\bm m,\bm m}_{2,1}\\
&L^{\bm \tau,\bm \tau}_{2,2},L^{\bm n,\bm \tau}_{2,2},L^{\bm m,\bm \tau}_{2,2},L^{\bm \tau,\bm n}_{2,2},L^{\bm n,\bm n}_{2,2},L^{\bm m,\bm n}_{2,2},L^{\bm \tau,\bm m}_{2,2},L^{\bm n,\bm m}_{2,2},L^{\bm m,\bm m}_{2,2},\\
&L^{\bm \tau,\bm \tau}_{3,1},L^{\bm n,\bm \tau}_{3,1},L^{\bm m,\bm \tau}_{3,1},L^{\bm \tau,\bm n}_{3,1},L^{\bm n,\bm n}_{3,1},L^{\bm m,\bm n}_{3,1},L^{\bm \tau,\bm m}_{3,1},L^{\bm n,\bm m}_{3,1},L^{\bm m,\bm m}_{3,1}\\
&L^{\bm \tau,\bm \tau}_{3,2},L^{\bm n,\bm \tau}_{3,2},L^{\bm m,\bm \tau}_{3,2},L^{\bm \tau,\bm n}_{3,2},L^{\bm n,\bm n}_{3,2},L^{\bm m,\bm n}_{3,2},L^{\bm \tau,\bm m}_{3,2},L^{\bm n,\bm m}_{3,2},L^{\bm m,\bm m}_{3,2},\\
&L^{\bm \tau,\bm \tau}_{4,1},L^{\bm n,\bm \tau}_{4,1},L^{\bm m,\bm \tau}_{4,1},L^{\bm \tau,\bm n}_{4,1},L^{\bm n,\bm n}_{4,1},L^{\bm m,\bm n}_{4,1},L^{\bm \tau,\bm m}_{4,1},L^{\bm n,\bm m}_{4,1},L^{\bm m,\bm m}_{4,1}\\
&L^{\bm \tau,\bm \tau}_{4,2},L^{\bm n,\bm \tau}_{4,2},L^{\bm m,\bm \tau}_{4,2},L^{\bm \tau,\bm n}_{4,2},L^{\bm n,\bm n}_{4,2},L^{\bm m,\bm n}_{4,2},L^{\bm \tau,\bm m}_{4,2},L^{\bm n,\bm m}_{4,2},L^{\bm m,\bm m}_{4,2},\\
&L^{\bm \tau,\bm \tau}_{5,1},L^{\bm n,\bm \tau}_{5,1},L^{\bm m,\bm \tau}_{5,1},L^{\bm \tau,\bm n}_{5,1},L^{\bm n,\bm n}_{5,1},L^{\bm m,\bm n}_{5,1},L^{\bm \tau,\bm m}_{5,1},L^{\bm n,\bm m}_{5,1},L^{\bm m,\bm m}_{5,1}\\
&L^{\bm \tau,\bm \tau}_{5,2},L^{\bm n,\bm \tau}_{5,2},L^{\bm m,\bm \tau}_{5,2},L^{\bm \tau,\bm n}_{5,2},L^{\bm n,\bm n}_{5,2},L^{\bm m,\bm n}_{5,2},L^{\bm \tau,\bm m}_{5,2},L^{\bm n,\bm m}_{5,2},L^{\bm m,\bm m}_{5,2},\\
&L^{\bm \tau,\bm \tau}_{6,1},L^{\bm n,\bm \tau}_{6,1},L^{\bm m,\bm \tau}_{6,1},L^{\bm \tau,\bm n}_{6,1},L^{\bm n,\bm n}_{6,1},L^{\bm m,\bm n}_{6,1},L^{\bm \tau,\bm m}_{6,1},L^{\bm n,\bm m}_{6,1},L^{\bm m,\bm m}_{6,1}\\
&L^{\bm \tau,\bm \tau}_{6,2},L^{\bm n,\bm \tau}_{6,2},L^{\bm m,\bm \tau}_{6,2},L^{\bm \tau,\bm n}_{6,2},L^{\bm n,\bm n}_{6,2},L^{\bm m,\bm n}_{6,2},L^{\bm \tau,\bm m}_{6,2},L^{\bm n,\bm m}_{6,2},L^{\bm m,\bm m}_{6,2},\\
&L_{1,i}^{1,f},i=1,2,\ L_{1,\bm \nu}^{1,f},\ L_{2,i}^{1,f},i=1,2,\ L_{2,\bm \nu}^{1,f},
\ L_{3,i}^{1,f},i=1,2, \ L_{3,\bm \nu}^{1,f},\ L_{4,i}^{1,f},i=1,2,\ L_{4,\bm \nu}^{1,f},\\
& L^{0,e}_{1,i},i=1,2,\cdots,7,\ L^{0,e}_{2,i},i=1,2,\cdots,7,\ L^{0,e}_{3,i},i=1,2,\cdots,7,\\
& L^{0,e}_{4,i},i=1,2,\cdots,7, \ L^{0,e}_{5,i},i=1,2,\cdots,7, \ L^{0,e}_{6,i},i=1,2,\cdots,7,\\
& L^{0,f}_{1,i},i=1,2,\cdots,15, \ L^{0,f}_{2,i},i=1,2,\cdots,15, \ L^{0,f}_{3,i},i=1,2,\cdots,15, \ L^{0,f}_{4,i},i=1,2,\cdots,15,\\
& L_i^{0,K},i=1,2,\cdots,20, \ L_i^{1,K},i=1,2,3,
\end{align*}
where $L_{\alpha,\bm \nu}^{1,f}=\int_{f_{\alpha}}\nabla\times\bm \phi\cdot\bm \nu_{\alpha}\mathrm{d}S,\ \alpha=1,2,3,4, \text{ and } \bm \nu_{\alpha} \text{ is the unit normal vector } \text{of the face }  f_{\alpha}.$

If we can find two matrices $D=(d_{ij})$ and $E=(e_{ij})$ such that 
\begin{align*}
	\tilde L_i = \sum_{j=1}^{383} d_{ij} L_j^*, \text{ in }\mathcal R_7^*,\ i = 1,2,\cdots,315,\\
	L^*_i = \sum_{j=1}^{315} e_{ij} L_j, \text{ in }\mathcal R_7^*,\ i = 1,2,\cdots,383,
\end{align*}	
then $C=DE$. By the transformation \eqref{mapping-curlu} and the chain rule, we have
\begin{align*}
	\quad \quad \begin{pmatrix}
	\tilde L^{v}_{\alpha,1}&
	\tilde L^{v}_{\alpha,2}&
	\tilde L^{v}_{\alpha,3}
\end{pmatrix}^T=\det(B_K)B_K^{-1}
\begin{pmatrix}
	 L^{*,v}_{\alpha,1}&
	 L^{*,v}_{\alpha,2}&
	 L^{*,v}_{\alpha,3}
\end{pmatrix}^T,
\end{align*}
\begin{align*}
	&\quad \quad \quad \begin{pmatrix}
	\tilde L^{x}_{\alpha,1}&
	\tilde L^{y}_{\alpha,1}&
	\tilde L^{z}_{\alpha,1}&
	\tilde L^{x}_{\alpha,2}&
	\tilde L^{y}_{\alpha,2}&
	\tilde L^{z}_{\alpha,2}&
    \tilde L^{x}_{\alpha,3}&
	\tilde L^{y}_{\alpha,3}
\end{pmatrix}^T\\
&~~=W
\begin{pmatrix}
	\tilde L^{*,x}_{\alpha,1}&
	\tilde L^{*,y}_{\alpha,1}&
	\tilde L^{*,z}_{\alpha,1}&
	\tilde L^{*,x}_{\alpha,2}&
    \tilde L^{*,y}_{\alpha,2}&
	\tilde L^{*,z}_{\alpha,2}&
    \tilde L^{*,x}_{\alpha,3}&
	\tilde L^{*,y}_{\alpha,3}&
		\tilde L^{*,z}_{\alpha,3}
\end{pmatrix}^T,
\end{align*}
\begin{align*}
&\left(\!\!\begin{array}{ccccccccccccccc}
	\tilde L^{yy}_{\alpha,1}\!&
	\tilde L^{zz}_{\alpha,1}\!&
	\tilde L^{xz}_{\alpha,1}\!&
	\tilde L^{xy}_{\alpha,1}\!&
	\tilde L^{yz}_{\alpha,1}\!&
	\tilde L^{xx}_{\alpha,2}\!&
	\tilde L^{zz}_{\alpha,2}\!&
	\tilde L^{xz}_{\alpha,2}\!&
	\tilde L^{xy}_{\alpha,2}\!&
    \tilde L^{yz}_{\alpha,2}\!&
    \tilde L^{xx}_{\alpha,3}\!&
	\tilde L^{yy}_{\alpha,3}\!&
	\tilde L^{xz}_{\alpha,3}\!&
	\tilde L^{xy}_{\alpha,3}\!&
	\tilde L^{yz}_{\alpha,3}
\end{array}\!\!\right)^T\\
&=V\left(\!\!\begin{array}{ccccccccccccc}
 L^{*,xx}_{\alpha,1}\!&
 L^{*,yy}_{\alpha,1}\!&
 L^{*,zz}_{\alpha,1}\!&
 L^{*,xz}_{\alpha,1}\!&
 L^{*,xy}_{\alpha,1}\!&
 L^{*,yz}_{\alpha,1}\!&
 \cdots \! &
  L^{*,xx}_{\alpha,3}\!&
 L^{*,yy}_{\alpha,3}\!&
 L^{*,zz}_{\alpha,3}\!&
 L^{*,xz}_{\alpha,3}\!&
 L^{*,xy}_{\alpha,3}\!&
 L^{*,yz}_{\alpha,3}
\end{array}\!\!\right)^T\!,
\end{align*}
where $W$ is the first 8 rows of the matrix $\det(B_K)B_K^{-1}\bigotimes B_K^T$ and $V$ is the matrix $\det(B_K)B_K^{-1}\bigotimes H$ without first, eighth, fifteen rows with
\begin{align*}
	H=\begin{pmatrix}
		B_{11}^2&B_{21}^2&B_{31}^2&2B_{11}B_{31}&2B_{11}B_{21}&2B_{31}B_{21}\\
		B_{12}^2&B_{22}^2&B_{32}^2&2B_{12}B_{32}&2B_{12}B_{22}&2B_{32}B_{22}\\
		B_{13}^2&B_{23}^2&B_{33}^2&2B_{13}B_{33}&2B_{13}B_{23}&2B_{33}B_{23}\\
		B_{11}B_{13}&B_{21}B_{23}&B_{31}B_{33}&B_{11}B_{33}+B_{13}B_{31}&B_{11}B_{23}+B_{13}B_{21}&B_{31}B_{23}+B_{33}B_{21}\\
		B_{11}B_{12}&B_{21}B_{22}&B_{31}B_{32}&B_{11}B_{32}+B_{12}B_{31}&B_{11}B_{22}+B_{12}B_{21}&B_{31}B_{22}+B_{32}B_{21}\\
		B_{12}B_{13}&B_{22}B_{23}&B_{32}B_{33}&B_{32}B_{13}+B_{33}B_{12}&B_{13}B_{22}+B_{12}B_{23}&B_{32}B_{23}+B_{33}B_{22}\\
	\end{pmatrix}.
\end{align*}
Similarly,
\begin{align*}
	\begin{pmatrix}
	\tilde L^{1,e}_{\alpha,1}\\
	\tilde L^{1,e}_{\alpha,2}\\
	\tilde L^{1,e}_{\alpha,3}
\end{pmatrix}=\det(B_K)B_K^{-1}
\begin{pmatrix}
	 L^{*,1,e}_{\alpha,1}\\
	 L^{*,1,e}_{\alpha,2}\\
	 L^{*,1,e}_{\alpha,3}
\end{pmatrix},
\end{align*}
If we represent $B_K^{-T}\hat{\bm \tau}_{\alpha}$ and $B_K\hat{\bm n}_{\alpha}$ by using the bases $\bm \tau_{\alpha},\bm n_{\alpha},\bm m_{\alpha}$, we'll get
\begin{align*}
	&\quad \quad \quad \quad \quad \quad \quad \quad \begin{pmatrix}
	\tilde L^{\bm\tau,\bm n}_{\alpha,i}\!&
	\tilde L^{\bm n,\bm n}_{\alpha,i}\!&
	\tilde L^{\bm m,\bm n}_{\alpha,i}\!&
	\tilde L^{\bm\tau,\bm m}_{\alpha,i}\!&
	\tilde L^{\bm n,\bm m}_{\alpha,i}\!&
	\end{pmatrix}^T\\
	&=G_{\alpha}
		\begin{pmatrix}
	L^{*,\bm\tau,\bm \tau}_{\alpha,i}\!&
	L^{*,\bm n,\bm \tau}_{\alpha,i}\!&
	 L^{*,\bm m,\bm \tau}_{\alpha,i}\!&
	 	 L^{*,\bm\tau,\bm n}_{\alpha,i}\!&
	 L^{*,\bm n,\bm n}_{\alpha,i}\!&
	 L^{*,\bm m,\bm n}_{\alpha,i}\!&
	 L^{*,\bm\tau,\bm m}_{\alpha,i}\!&
	 L^{*,\bm n,\bm m}_{\alpha,i}\!&
	 L^{*,\bm m,\bm m}_{\alpha,i} \!&
\end{pmatrix} ^T
\end{align*}
with
\begin{align*}
G_{\alpha}=\det(B_K)
 		\begin{pmatrix}
		\big([\bm \tau_{\alpha},\bm n_{\alpha},\bm m_{\alpha}]B^{-T}_K\hat{\bm \tau}_{\alpha}\big)\bigotimes \big([\bm \tau_{\alpha},\bm n_{\alpha},\bm m_{\alpha}]B_K\hat{\bm n}_{\alpha}\big)\\
		\big([\bm \tau_{\alpha},\bm n_{\alpha},\bm m_{\alpha}]B^{-T}_K\hat{\bm n}_{\alpha}\big)\bigotimes \big([\bm \tau_{\alpha},\bm n_{\alpha},\bm m_{\alpha}]B_K\hat{\bm n}_{\alpha}\big)\\
		\big([\bm \tau_{\alpha},\bm n_{\alpha},\bm m_{\alpha}]B^{-T}_K\hat{\bm m}_{\alpha}\big)\bigotimes \big([\bm \tau_{\alpha},\bm n_{\alpha},\bm m_{\alpha}]B_K\hat{\bm n}_{\alpha}\big)\\
		\big([\bm \tau_{\alpha},\bm n_{\alpha},\bm m_{\alpha}]B^{-T}_K\hat{\bm \tau}_{\alpha}\big)\bigotimes \big([\bm \tau_{\alpha},\bm n_{\alpha},\bm m_{\alpha}]B_K\hat{\bm m}_{\alpha}\big)\\
		\big([\bm \tau_{\alpha},\bm n_{\alpha},\bm m_{\alpha}]B^{-T}_K\hat{\bm n}_{\alpha}\big)\bigotimes \big([\bm \tau_{\alpha},\bm n_{\alpha},\bm m_{\alpha}]B_K\hat{\bm m}_{\alpha}\big)\\
	\end{pmatrix}^T
 \end{align*}
 Similarly,
\begin{align*}
	\begin{pmatrix}
	\tilde L^{1,f}_{\alpha,1}\\
	\tilde L^{1,f}_{\alpha,2}\\
	\end{pmatrix}=
	C_{\alpha}
	\begin{pmatrix}
	L^{*,1,f}_{\alpha,1}\\
	L^{*,1,f}_{\alpha,2}\\
	 L^{*,1,f}_{\alpha,\bm n}\\
	\end{pmatrix} \text{ with } C_{\alpha}=\det(B_K)
	\begin{pmatrix}
		\big([\bm \tau^{\alpha}_1,\bm \tau^{\alpha}_2,\bm n_{\alpha}]^{-1}B^{-T}_K\hat{\bm \tau}^{\alpha}_1\big)^T\\
		\big([\bm \tau^{\alpha}_1,\bm \tau^{\alpha}_2,\bm n_{\alpha}]^{-1}B^{-T}_K\hat{\bm \tau}^{\alpha}_2\big)^T\\
			\end{pmatrix}
			\end{align*}
Since $L^{0,e}_{\alpha,i}=\hat L^{0,e}_{\alpha,i}$, $L^{0,f}_{\alpha,i}=\hat L^{0,f}_{\alpha,i}$,$L^{1,K}_{\alpha,i}=\hat L^{1,K}_{\alpha,i}$, and $L^{0,K}_{\alpha,i}=\hat L^{0,K}_{\alpha,i}$, we have
\begin{align*}
	&\tilde L^{0,e}_{\alpha,i}=L^{*,0,e}_{\alpha,i},\
	\tilde L^{0,f}_{\alpha,i}=L^{*,0,f}_{\alpha,i},\\
	&\tilde L^{0,K}_{\alpha,i}=L^{*,0,K}_{\alpha,i},
	\tilde L^{1,K}_{\alpha,i}=L^{*,1,K}_{\alpha,i}.
\end{align*}
	
Now we are in the position to giving the explicit expression of D,
\begin{align*}
	D=& \text{diag}(\det(B_K)B_K^{-1},\det(B_K)B_K^{-1},\det(B_K)B_K^{-1},\det(B_K)B_K^{-1},W,W,W,W,V,V,V,V,\\
&\det(B_K)B_K^{-1},\det(B_K)B_K^{-1},\det(B_K)B_K^{-1},\det(B_K)B_K^{-1},\det(B_K)B_K^{-1},\det(B_K)B_K^{-1},\\
&G_1,G_1,G_2,G_2,G_3,G_3,G_4,G_4,G_5,G_5,G_6,G_6,C_1,C_2,C_3,C_4,I_{125\times125}).
\end{align*}

Next, we express $L^*$ by  $L$ and some other functionals which actually can also be represented by $L$.

It's trivial to represent $L^{*,z}_{\alpha,i},L^{*,xx}_{\alpha,1},L^{*,yy}_{\alpha,2},L^{*,zz}_{\alpha,3}$. Now we show how to represent $L^{*,\bm\tau,\bm\tau}$. Let $\varphi(t)=(\nabla\times\bm\phi\cdot\bm \tau)(t\bm x_{\beta}+(1-t)\bm x_{\alpha})\in P_6(t).$
Then 
\begin{align*}
	&L^{*,\bm\tau,\bm\tau}_{\gamma,1}\bm \phi=\varphi\prime(\frac{1}{3})
	=-\frac{248}{81}\varphi(0)-\frac{8}{81}\varphi(1)+\frac{256}{81}\varphi(1/2)-\frac{40}{81}\varphi\prime(0)+\frac{1}{81}\varphi\prime(1)-\frac{2}{81}\varphi\prime\prime(0),\\
	&L^{*,\bm\tau,\bm\tau}_{\gamma,2}\bm \phi=\varphi\prime(\frac{2}{3})
	=\frac{8}{81}\varphi(0)+\frac{248}{81}\varphi(1)-\frac{256}{81}\varphi(1/2)+\frac{1}{81}\varphi\prime(0)-\frac{40}{81}\varphi\prime(1)+\frac{2}{81}\varphi\prime\prime(1).
\end{align*}
If we represent $\varphi(0),\varphi(1),\varphi'(0),\varphi'(1),\varphi(1/2),\varphi''(0)$, and $\varphi''(1)$ by some $L_i$, we can obtain $L^{*,\bm\tau,\bm\tau}_{\gamma,i}$ in the form of a linear combination of $L_i,i=1,2,\cdots,315$.
Similarly, we can represent $L^{*,\bm n,\bm \tau}$ and $L^{*,\bm m,\bm \tau}$. As for $L^{*,\bm m,\bm m}$, we can find 8 constants such that
\begin{align*}
\nabla(\nabla\times\bm \phi\cdot\bm m)\cdot\bm m=& C_1\nabla(\nabla\times\bm \phi\cdot\bm \tau)\cdot\bm \tau+C_2\nabla(\nabla\times\bm \phi\cdot\bm \tau)\cdot\bm n+C_3\nabla(\nabla\times\bm \phi\cdot\bm \tau)\cdot\bm m \\
&+C_4\nabla(\nabla\times\bm \phi\cdot\bm n)\cdot\bm \tau+C_5\nabla(\nabla\times\bm \phi\cdot\bm n)\cdot\bm n+C_6\nabla(\nabla\times\bm \phi\cdot\bm n)\cdot\bm m \\
&+C_7
\nabla(\nabla\times\bm \phi\cdot\bm m)\cdot\bm \tau+
C_8\nabla(\nabla\times\bm \phi\cdot\bm m)\cdot\bm n,
\end{align*}
since $\nabla\cdot(\nabla\times\bm \phi)=0$.
Furthermore, $\nabla(\nabla\times\bm \phi\cdot\bm \tau)\cdot\bm \tau$, $\nabla(\nabla\times\bm \phi\cdot\bm n)\cdot\bm \tau$ and $\nabla(\nabla\times\bm \phi\cdot\bm \tau)\cdot\bm m$ can be determined by the values of $\nabla\times\bm\phi$ and its up to second derivatives at two endpoints since they are in $P_6$. So far, we can represent $L^{*,\bm m,\bm m}$ in terms of $L$.

Since $(\nabla\times\bm \phi\cdot\bm n,1)_f=\langle\bm \phi\cdot\bm \tau,1\rangle_{\partial f},$
we can represent $L^{*,1,f}$ as a linear combination of $L^{0,e}$, and hence $L_i$.

Finally, we can express $L^*$ by $L$. 
Then we obtain $C=DE$. Because of the large number of degrees of freedom, it's tedious to implement this process in Matlab. We provide the code for basis functions at \\
 \url{https://github.com/QianZhangMath/3D-curl-curl-conforming-FE}.

\section*{Acknowledgments}
We would like to thank Professor Jiguang Sun for drawing our attention to the quad-curl problem. We would also like to thank Professor Jiguang Sun and Professor Huiyuan Li for their valuable comments and suggestions.

\bibliographystyle{plain}
\bibliography{quadcurl-3d}{}
~\\
\end{document}